\documentclass[11pt,reqno]{article} 

\usepackage{tikz}
\usepackage{lmodern}
\usepackage{amsmath,amsfonts,latexsym,amscd,amssymb, amsthm}
\usepackage{bigints,moreverb,rotating,graphics,amsfonts,latexsym,amscd}
\usepackage{hyperref}
\usepackage{blindtext}
\usepackage{amsfonts}
\usepackage{enumitem}
\hypersetup{colorlinks=true,allcolors=blue}

\def\R{\mathbb{R}}
\def\g{\gamma}

\def\<{\langle}
\def\>{\rangle}
\def\a{\alpha}
\def\eps{\varepsilon}

\def \o{\Omega}
\def \oo{\omega}
\def \d{\delta}
\def\e{\varepsilon}
\newtheorem{thm}{Theorem}[section]
\newtheorem{prop}[thm]{Proposition}
\newtheorem{lem}[thm]{Lemma}

\newtheorem{theorem}[thm]{Theorem}
\newtheorem{definition}[thm]{Definition}
\newtheorem{proposition}[thm]{Proposition}
\newtheorem{corollary}[thm]{Corollary}
\newtheorem{rem}[thm]{Remark}
\DeclareMathOperator*{\esssup}{\esssup}
\begin{document}
	
	\title{\bf Stochastic homogenization of Hamilton-Jacobi equations on a junction}
	\author{
		\normalsize\textsc{FORCADEL Nicolas$^1$, FAYAD Rim$^1$, IBRAHIM Hassan $^2$ }} \vspace{20pt} \maketitle
	\footnotetext[1]{Normandie Univ, INSA de Rouen, LMI (EA 3226 - FR CNRS 3335), 76000 Rouen, France, 685 Avenue de
		l'Universit\'e, 76801 St Etienne du Rouvray cedex. France}    
	\footnotetext[2]{	Lebanese University, Faculty of Sciences-I Mathematics Department Hadath, Beyrouth, Lebanon}

	\maketitle
	\begin{center}
		\Large{Abstract} 
		
	\end{center}
	{\indent{\footnotesize{\quad We consider the specified stochastic homogenization of first order evolutive Hamilton-Jacobi equations on a very simple junction, i.e the real line with a junction at the origin. Far from the origin, we assume that the considered hamiltonian is closed to given stationary ergodic hamiltonians (which are different on the left and on the right). Near the origin, there is a perturbation zone which allows to pass from one hamiltonian to the other. The main result of this paper is a stochastic homogenization as the length of the transition zone goes to zero. More precisely, at the limit we get two deterministic right and left hamiltonians with a deterministic junction condition at the origin. The main difficulty and novelty of the paper come from the fact that the hamiltonian is not stationary ergodic. Up to our knowledge, this is the first specified stochastic homogenization result. This work is motivated by traffic flow applications.    }}}
	
	\section{General Introduction}
	
	In this paper we are interested in the specified stochastic homogenization of first order evolutive stochastic Hamilton-Jacobi equations posed on the real line. Our work is motivated by the paper of Imbert, Galise and Monneau \cite{Galise} which studies the specified (periodic) homogenization of an Hamilton-Jacobi equation with some applications in traffic flow. 	
	
	Many homogenization's results concerning the rescaling of traffic's dynamics have been obtained in the periodic setting. We refer for instance to \cite{homoge.periodique,forcadel, FIM,FIM2,loi de conse} for micro-macro passage. In this setting, all the drivers are assumed to be identical (or if we have different types of drivers, they are periodically distributed). From a modeling point of view, this assumption allows to get very interesting results and to justify macroscopic models but it is not very realistic. In this paper, we investigate the stochastic setting in which the type of drivers are randomly distributed. A first result concerning the stochastic case were obtained recently in \cite{Card-For} and concerns the micro-macro passage in a single road with a stochastic distribution of the drivers. In this paper, we consider a stochastic hamiltonian describing the traffic on a single road (i.e the real line). The main difficulty and novelty come from the fact that, at the origin, we assume that there is a local perturbation. The typical examples we have in mind are a speed traffic sign, a slowdown near a school or due to a car crash near the road for example). Then we assume that far from the origin, the hamiltonian is closed to two given hamiltonian (one on the left and the other on the right) and near the origin, there is a perturbation and a transition zone which allows to pass from one hamiltonian to the other. 	 The goal of our paper is to understand which is the limiting model as the length of the transition zone goes to zero. At the limit we will recover two (deterministic) hamiltonians on the left and on the right, with a (deterministic) junction condition in the origin.

	The stochastic homogenization of Hamilton-Jacobi equations has been extensively studied since the pioneer work of Souganidis \cite{souga}. We refer for example to  {\cite{homog.stoch.quantitative.premier.ordre.concexe,homo.stoch.premier.ordre.quasiconvex,homo.stoch.premier.ordre.convexe} for the case of first-order convex and quasiconvex Hamilton-Jacobi equation, to \cite{Second.order.equation.homoge.stoch,Second.order.equation.homoge.stoch.2,Homo.stoch.second.ordre,homo.stoch.second.ordre.3} for second order convex and non-convex Hamilton-Jacobi equation, and to \cite{Homoge.stoch.pour.un.hamiltonian.qui.depend.du.temps,Homoge.stoch.pour.un.hamiltonian.qui.depend.du.temps.2} for the case of hamiltonian depending on time. Concerning the specified homogenization in the periodic case, we refer to \cite{achdou,Galise, Lions}.
	Nevertheless, up to our knowledge, this is the first result concerning specified stochastic homogenization. The main difficulty in this setting comes from the fact that the hamiltonian is no more stationary ergodic, which is an essential assumption in the rest of the literature. To overcome this difficulty, we will use some technics developed in \cite{homog.stoch.quantitative.premier.ordre.concexe}
to obtain quantitative homogenization result's.

		\section{Assumptions and main results}\label{assum and main result}
		We begin this section with several assumptions, on a hamiltonian $H=H(p,y,\oo)$,  needed to prove our stochastic homogenization result. First we want to define the probability space, so we define for all $ V $ $ \in \mathbf {B} $, where $\mathbf{B}$ denoted the Borel $\sigma$-algebra on $\R$, the $\sigma$-algebra $\mathbf F(V)$ by $$ \mathbf F(V):= \{\sigma \mbox{- algebra generated by the maps:} ~ \oo \rightarrow H(p,y,\oo), \forall~y \in V, p \in \R\}.$$ Let $\mathbf F=\mathbf F(\R)$ and  $(\o, \mathbf{F},P)$ be a probability space. 
		
		For all $\omega \in \o$, we consider the equation, for all $T>0$, $$u_{t}(t,y)+H(Du,y,\omega)=0 ~~\mbox{in}~ (0,T)\times\R,$$where the hamiltonian $H:\R\times\R\times \o \to \R$ is measurable with respect to $\mathbf{B} \times \mathbf{B} \times \mathbf{F} $.
		
		We make the following assumptions on $H$.
		\begin{enumerate}[font={\bfseries},label={(H\arabic*)}]
\item \underline{Regularity with respect to $p$}:  $H$ is Lipschitz continuous with respect to $p$ uniformly in $(y,\oo)$.
\item \underline{Boundedness of $H$ in $(p,y)$}:
			There exists $c_{0},C_{0},\gamma>0$, such that for all $p, y \in \R,$ and $\oo \in \o$ we have $$ -c_{0}|p+\gamma|\le H(p,y,\oo)\le C_{0}|p+\gamma| .  $$ In particular, we obtain that 	
			for every $R>0$, the family of functions $\{H(\cdot,\cdot,\oo);\oo\in \o\}$ is bounded for $(p,y)\in B_{R} \times \R$, with $ B_{R}$ a ball of radius $R$ in $\R$.
\item \underline{Uniform coercivity}: $H$ is coercive with respect to $p$, uniformly in $(y,\oo)$, i.e $$ \lim\limits_{|p|\to +\infty} \inf\limits_{(y,\oo)} H(p,y,\oo) =+\infty.$$
\item \underline{Uniform modulus of continuity in space}: There exists a modulus of continuity $w$ such that $\forall x,y,p\in \R$ and $\oo \in \o$ we have $$|H(p,y,\oo)-H(p,x,\oo)| \le w(|x-y|(1+|p|)).$$
\item \underline{Convexity}: For all $(y,\oo)$, the hamiltonian $p \to H(p,y,\oo)$ is convex.
\item \underline{A unique lower bound for $H$:}\label{inegalite du controle optimal sur H} For all $y,p \in \R$, and $\oo \in \o$ we have \begin{equation}
			H(p,y,\oo)\ge H(0,y,\oo).	    
			\end{equation} 
			\end{enumerate}
			\begin{rem}
			We can consider a hamiltonian given by traffic flow modelling (see \cite{Imbert}) such that, there exists $p_{0}\in \R$ such that  for all $y \in \R,\oo \in \o$ we have that $$ H(p,y,\oo)\ge H(p_{0},y,\oo).  $$ We obtain the same results by defining a hamiltonian $\tilde{H}$ given by $$  \tilde{H}(p,y,\oo)=H(p-p_{0},y,\oo) $$ which verifies ${\bf (H6)}$.
						\end{rem}
\begin{enumerate}[font={\bfseries},label={(H\arabic*)}]
\setcounter{enumi}{6}
\item \underline{Left and right hamiltonians}: There exists two stochastic hamiltonians $H_{L}$ and $H_{R}$ such that $H$ is equal to $H_L$ near $-\infty$ and to $H_R$ near $+\infty$. More precisely, we assume either:

\noindent $\textbf{(H7}${\bf -}$\textbf{WFL)}$
The hamiltonian $H$ can be written as follow 
\begin{equation}\label{forme restrictive de H}
H(p,y,\oo)=\phi(y)H_{L}(p,y,\oo)+(1-\phi(y))H_{R}(p,y,\oo),	
\end{equation} 
where $\phi(\cdot)$ is a non-increasing function in $C^{\infty}(\R)$, given by
$$ \phi(y)=\left\{
\begin{array}{ll}
1 & \mbox{if}~y\le -1 ,\\
0 & \mbox{if}~y \ge 1.
\end{array}
\right. 
$$
\end{enumerate}
\begin{enumerate}[font={\bfseries},label={(H\arabic*)}]
\setcounter{enumi}{7}
\item[or] ~\newline

\noindent $\textbf{(H7}${\bf -}${\bf \eps}${\bf )} There exists an Hamiltonian $H_{0}\ge H_R, H_L$
such that the Hamiltonian $H$ is given, for $\eps>0$, by
\begin{equation}\label{eq:2-3}
 H(p,y,\oo)=\psi^{\eps}_{1}(-y)H_{L}(p,y,\oo)+H_{0}(p,y,\oo)(1-\psi^{\eps}_{1}(-y))(1-\psi^{\eps}_{1}(y) )+\psi^{\eps}_{1}(y)H_{R}(p,y,\oo),   
\end{equation} 
with $\psi_1^\eps\in C^{\infty}(\R)$ satisfying
$$  \psi_{1}^{\eps}(y)=\left\{
\begin{array}{ll}
1 & y > \frac{2}{\sqrt{ \eps}}\\
0& y <\frac{1}{\sqrt{ \eps}}.
\end{array}
\right.$$ 
\end{enumerate}
\begin{rem}
In ${\bf (H7}${\bf -}${\bf \eps)}$, the Hamiltonian $H$ depends on $\eps$. Nevertheless, to simplify the notation, we don't explicit this dependence. Note also that the power of $\eps$ ($-1/2$ here) is fixed  to simplify the presentation, but one can replace $\eps^{-\frac 12}$ by  $\eps^{-\a}$, with $\a\in (0,1)$, in the definition of $\psi_1^\eps$ and the homogenization result will still hold true.
\end{rem}
\begin{rem}
Let us give some explanations on Assumptions ${\bf (H7}${\bf -}${\bf WFL)}$ and  ${\bf (H7}${\bf -}${\bf \eps)}$. In fact, in  Assumption ${\bf (H7}${\bf -}${\bf WFL)}$, the hamiltonian pass in a convex way from $H_L$ to $H_R$ and we will see that in that case the flux is not limited. The typical example we have in mind is  a speed reduction or increase near the origin. On the contrary, in Assumption  ${\bf (H7}${\bf -}${\bf \eps)}$, the speed is reduced near the origin. The typical example we have in mind is $H_0=CH_L$, with $C<1$ (in traffic modelling, hamiltonians are negative) which represents a slowdown near the origin (near a school or due to a car crash near the road for example). In that case, we will see that the flux is limited by this slowdown. Let us point out that in the case where the flux is limited, we have to consider a perturbation zone of radius $1/\sqrt \eps$. After the rescaling, we will see that the radius of the perturbation will be $\sqrt \eps$. This radius is necessary to get the convergence of $u^\eps$ to a deterministic function. In particular, we will present a counter-example in Section \ref{sec:42} showing that if $\psi_1^\eps$ doesn't depend on $\eps$ then, the limit can't be deterministic.

\end{rem}
The probability space $(\Omega, \mathbf{F}, P)$ is endowed with an ergodic group $(\tau_{z})_{z \in \R}$ of $\mathbf{F}-$measurable, measure-preserving transformations $\tau_{z}:\o \to \o$. That is, we assume that for every $x,z \in \R$ and $A \in \mathbf{F}$  $$  \tau_{x+z}=\tau_{x}\circ \tau_{z}~\mbox{and}~P(\tau_{z}(A))=P(A), $$ and $$ \mbox{if}~\tau_{z}(A)=A, ~\forall z \in \R, ~\mbox{then either} ~P(A)=0~\mbox{or}~P(A)=1.$$ 
			\begin{enumerate}[font={\bfseries},label={(H\arabic*)}]
			\setcounter{enumi}{7}
			\item \underline{Stationarity of $ H_{\alpha} $}: The hamiltonians $H_{\a}$, for $\a=L,R,0$, are stationary in $(y,\oo)$ with respect to the group $ (\tau_{z})_{z\in \R}$, that is, for every $p,y,z\in \R$ and $\oo \in \o$, 
$$H_{\a}(p,y+z,\omega)=H_{\a}(p,y,\tau_{z}(\omega)).$$
\item \underline{Ergodicity}: 
The probability $P$ satisfies a unit range of dependence:  
$$
\mbox{For all} ~U,V \in \mathbf{B} ~\mbox{such that}~ d(U,V)\ge 1, ~\mbox{we have that}~  \mathbf F(U) ~\mbox{and}~ \mathbf F(V)~ \mbox{are independent. }
$$
\end{enumerate}
For simplicity of notation, we call ${\bf (H)}$ the set of assumptions ${\bf(H1)}$-${\bf (H9)}.$ If we want to specify which assumption is satisfied in ${\bf (H7)}$, we denote it by ${\bf (H}$-${\bf WFL)}$ or ${\bf (H}$-${\bf \eps)}$.
	
\begin{rem}
Let us note that, even if $H_R, H_L, H_0$ are stationary ergodic, due to its particular form (see \eqref{forme restrictive de H} or \eqref{eq:2-3}), $H$ is no longer stationary ergodic.
\end{rem}		

We now give an example of a Hamilton-Jacobi equation, inspired by traffic flow modelling, which verifies ${\bf (H)}$. Let $V_\a:\R^{+}\to \R^{+}$, $\a=L,R,0$, and $\psi:\R\times \o\to \R $ be four functions satisfying the following assumptions:
\begin{itemize}
\item $V_\a$ is Lipschitz continuous and positive.
\item $V_\a$ is non-decreasing and $\lim\limits_{h \to +\infty}V_\a(h)=V_{max}^\a<+\infty$.
\item There exists $h_{0}^\a$ such that $V_\a(h)=0$ for all $h \le h_{0}^\a$.
\item There exists $\tilde{p}\in [-1/h_{0}^\a,0)$ such that the function $p\to p V_\a(-\frac{1}{p})$ is non-increasing in $[ -1/h_{0}^\a;\tilde{p})$ and non-decreasing in $[\tilde{p},0) $.
\item $\psi $ is stationary in  $(y, \oo)$ with respect to the translation group $(\tau_{z})_{z}$, that is, we assume that, for every $y,z\in \R$ and $\oo \in \o$ we have $$ \psi(y,\tau_{z}\oo)=  \psi(y+z,\oo).$$
\item $\psi$ is a positive, Lipschitz continuous function with respect to $y$.
\item $\psi $ is bounded for all $y, \oo$.
\item For each $U\in \textbf{B}$, we denoted by $\textbf{F}(U)$ the $\sigma$-algebra generated by the sets $\{   \oo \to \psi(y,\oo); y \in U\}$. We have that for $U, V \in \textbf{B}$ such that $d(U,V)\ge 1$, $\textbf{F}(U)$ and $\textbf{F}(V)$ are independent.
\end{itemize}

A typical hamiltonian for traffic flow is given by 
		$$
		H^*_{\a}(p )=\left\{
		\begin{aligned}
		&(-p-\tilde{p}-k_{0}), &p<-k_{0}-\tilde{p},\\
		&-|p+\tilde{p}|V_{\a}\left(-\frac{1}{p+\tilde{p}}\right), & -k_{0}-\tilde{p}\le p\le-\tilde{p},\\
		&(p+\tilde{p}) , &p>-\tilde{p},
		\end{aligned}
		\right.
		$$
where $k_0=1/h_0^{\a}$. We then add a stochastic perturbation by considering 
$$H_\a(p,y,\omega)= H^*_\a(p)  \psi(\oo,y)$$
which satisfied assumptions {\bf (H)}. In this setting, the variable $y$ represents a (continuous index of) vehicle and so the perturbation $\psi$ explains how the velocity of this vehicle is modified (for example if the vehicle $y$ is a car or a truck).

		\subsection{Main result} The main result of this paper is a stochastic homogenization result. 
		We  consider the following rescaled problem, for $T>0$,  
		\begin{equation}\label{P.en.epsilon}
		\left\{
		\begin{aligned}
		&u^{\eps}_{t}+H\left(Du^{\eps},\frac{y}{\eps},\oo\right)=0  & (t,y)& \in (0,T)\times \R, \\
		& u^{\eps}(0,y,\oo)=u_{0}(y) & y& \in \R,
		\end{aligned}
		\right.
		\end{equation} where the hamiltonian $H$ satisfies assumptions ${\bf (H)}$, and $u_{0}$ is a Lipschitz continuous function.
\begin{rem}
Under Assumption ${\bf (H}$-${\bf \e)}$, the radius of the perturbation zone becomes of order $\sqrt \e$ and will disappear as $\e\to 0$. Nevertheless, we will see that the macroscopic model we obtain below will keep into memory the presence of this perturbation via the flux limiter $\overline A$ at the origin.
\end{rem}
		We also consider the deterministic limit problem given by  
		\begin{equation}\label{P.jonction.limite}
		\left\{
		\begin{aligned}
		&u_{t}+\overline{H}_{L}(Du)=0,    & y<0&, t\in (0,T),\\
		&u_{t}+\overline{H}_{R}(Du)=0,    & y>0&, t\in (0,T),\\
		&u_{t}+\max(\overline{A},\overline{H}_{L}^{+}(Du(0^{-}),\overline{H}_{R}^{-}(Du(0^{+}))=0,    &y=0 &,t\in (0,T),\\
		&u(0,x)=u_{0}(y), &y \in \R,
		\end{aligned} 
		\right.
		\end{equation}
		where $\overline{H}_{\alpha}^{\pm}$ are the increasing and decreasing part of $\overline{H}_{\alpha}$ and given by the following definition
		\begin{definition}[Definition of the increasing and decreasing part of the hamiltonians $\overline{H}_{\alpha}$] \label{partiecroissantedecroi} For $\alpha\in \{R,L\} $, $\overline{H}_{\alpha}^{+}$ denotes the non-decreasing part of $\overline{H}_{\alpha}$, defined by
			$$\overline{H}_{\alpha}^{+}(p)=
			\left\{
			\begin{aligned}
			&\overline{H}_{\alpha}(p) &\mbox{if }& \, p \ge 0\\
			&\underline{H}_{\alpha} &\mbox {if }& \, p \le 0,
				\end{aligned}
			\right.$$ 
and  $\overline{H}_{\alpha}^{-}$ denotes the non-increasing part of $\overline{H}_{\alpha}$, defined by
			$$\overline{H}_{\alpha}^{-}(p)=
			\left\{
			\begin{aligned}
			&\overline{H}_{\alpha}(p) &\mbox{if }& \,p \le 0\\
			&\underline{H}_{\alpha} &\mbox{ if }& \, p \ge 0, 
				\end{aligned}
			\right.$$
			where $\underline{H}_{\a}=\min\limits_{p}(\overline{H}_{\alpha}(p))$.
		\end{definition}	
		In model \eqref{P.jonction.limite}, the constant $\overline A$ is called the flux limiter and explains how the flux is reduced at the origin (by the presence of the perturbation). We refer to \cite{IMZ} for more explanations on this flux limiter.
		
		The goal of this paper is to prove that the solution of \eqref{P.en.epsilon} converges to the solution of \eqref{P.jonction.limite} as $\e\to 0$. The homogenization result is given in the following theorem 
		\begin{thm}[Stochastic homogenization result]\label{Theoreme d homog du prob en eps}
			Assume that ${\bf (H)}$ is satisfied, $u_{0}$ is a Lipschitz continuous function and for $\eps >0$, let $u^{\eps}$ be the solution of \eqref{P.en.epsilon}. Then there exists two deterministic, convex and coercive hamiltonians $\overline{H}_{L}$ and $\overline{H}_{R}$,  a deterministic flux limiter $\overline A$ %
			and an event $\o_{0} \subseteq \o$ of full probability, such that for each $\oo \in \o_{0}$, the unique solution $u^{\eps}(\cdot,\cdot,\oo)$ of \eqref{P.en.epsilon} converges locally uniformly in $\R$, as $\eps\to 0$, to the unique deterministic solution $u$ of \eqref{P.jonction.limite}.
		\end{thm}
		
		\begin{rem}
		If Assumption ${\bf (H7}$-${\bf WFL)}$ is satisfied, we will see that 
$$\overline{A}=\max(\min\limits_{p}\overline{H}_{R}(p),\min\limits_{p}\overline{H}_{L}(p)).
$$
 From a traffic point of view, this means that the flux is not limited by the perturbation. On the contrary if ${\bf (H7}$-${\bf \eps)}$ is satisfied then 
 $$\overline{A}=\min\limits_p \overline H_0(p)\ge \max(\min\limits_{p}\overline{H}_{R}(p),\min\limits_{p}\overline{H}_{L}(p))
 $$ 
 and the flux can be limited.
		\end{rem}
		It is well known that to obtain a homogenization result, one have to introduce the so-called correctors. In the periodic case, the corrector is defined on the cell domain (see \cite {Evans, Evans2, Lions1}), whereas in the stochastic case, this corrector is defined on the wholedomain. Articles \cite{les.conditions.pour.l'existence.du.correcteur,existence ou non des correcteurs} prove that, in the general stochastic framework, the corrector does not exist. To solve this problem, we will classically use a sub-additive quantity which is a solution of an intermediate problem called the metric problem, see \cite {probleme.metrique}. To obtain the homogenization's result, it suffices to prove that the metric problem homogenizes.

		In this article, the sequence $\o_{0} \subseteq \o_{1} \subseteq\o_{2} \subseteq \dots$ is a sequence of subsets of $\o$ of full probability used to keep track of almost sure statements.
		\subsection{Organization of the paper} This article is organized as follow. In Section \ref{metric problem loins de zero}, we introduce the metric problems far from the junction and we recall some classical results about stochastic homogenization. Section \ref{metric problem sur la jonction} is devoted to the main results for the metric problem at the junction and in particular to the definition of the deterministic flux limiter. The proof of the homogenization's result for the metric problem at the junction is given in Section \ref{resultat d homogenization du probleme metrique}. In Section \ref{les resultats du correcteur app}, we give the main result for the approximated correctors and we prove in particular that the rescaled approximated correctors have the good slopes at infinity. Finally, in Section \ref{converg resultat} we conclude with the proof of the convergence result for the rescaled problem.    
		
		\section{Metric problem far from the origin}\label{metric problem loins de zero}
		In this section, we give some useful results obtained in \cite{homo.stoch.premier.ordre.quasiconvex,homo.stoch.premier.ordre.convexe} for an ergodic stationary hamiltonian.
		To define the deterministic hamiltonians $\overline{H}_{L}$ and $\overline{H}_{R}$, we consider two metric problems, defined by the left and right hamiltonians $H_{L} $ and $H_{R} $, as follows
		\begin{equation}\label{P.M.LOIN.J.avec.p}\left\{
		\begin{aligned}
		&H_{\a}\left(Dv(y,\oo)+p,y,\oo\right)=\mu,  & y \in \R \setminus \{0\}, \\
		&v(0,\oo)=0,
		\end{aligned}
		\right.
		\end{equation}with $\a \in \{L,R\}$. Recent works in stochastic homogenization prove that the set of solutions of \eqref{P.M.LOIN.J.avec.p} is included in $$ S'=\{    \phi; ~\mbox{such that}~\phi\in  \mathbf{Lip} ~\mbox{and strictly sub-linear}\}, $$ where $\mathbf{Lip}$ denotes the set of real-valued global Lipschitz functions on $\R$. This result is proven in \cite{homo.stoch.premier.ordre.quasiconvex,homo.stoch.premier.ordre.convexe}.
		\begin{proposition}[{{Existence of a solutions for the metric problems \cite[Proposition 3.2]{homo.stoch.premier.ordre.quasiconvex}}}]\label{pro:3.1}
			We set, for $\oo\in \o$, $$\mu^{\star}_{\alpha}(\oo) = \inf\{\mu;\mbox{ s.t there exists a function}~v \in \mathbf{Lip} ~\mbox{which satisfies}~H_{\a}(Dv,y,\oo) \le \mu \mbox{ in } \R\}. $$We have by stationarity and ergodicity, that $ \mu^{\star}_{\alpha}(\oo)=\mu^{\star}_{\alpha}$ almost surely on $\oo$. For all $\mu \ge \mu^{\star}_{\alpha}$, there exists a Lipschitz continuous solution $m_{\mu}^{\a}$ of\begin{equation}\label{P.M.LOIN.J}\left\{
			\begin{aligned}
			&H_{\a}\left(Dm_{\mu}^{\a}(y,0,\oo),y,\oo\right)=\mu  & y \in \R \setminus \{0\}, \\
			&m_{\mu}^{\a}(0,0,\oo)=0.
			\end{aligned}
			\right.
			\end{equation}For all $\mu >\inf\limits_{v \in S'}\sup\limits_{y \in \R}H_{\a}(Dv+p,y,\oo)$, the function $m_{\mu}^{\a}-py$ is the unique solution of \eqref{P.M.LOIN.J.avec.p} in $S'_{+}$, where $$S'_{+}=\{    \phi; ~\mbox{such that}~\phi \in \mathbf{Lip};~\liminf_{|y| \to +\infty}|y|^{-1}\phi(y)\ge 0 \}.$$
		\end{proposition}
		\begin{proposition}[Properties of $m_{\mu}^{\a}$ {{\cite[Proposition 3.2]{homo.stoch.premier.ordre.quasiconvex}}}]
			The solution $m_{\mu}^{\a}$ is the maximal sub-solution of \eqref{P.M.LOIN.J} in the sense that, if $v$ is a Lipschitz continuous function and is sub-solution of $H_{\alpha}(Dv,y,\oo)\le\mu  $ in $\R$ then $$v(\cdot,\oo)-v(0,\oo)\le m_{\mu}^{\a}(\cdot,0,\oo)~\mbox{in}~\R .$$ 
Moreover, $m_\mu^\a$ is sub-additive i.e for every $x,y,z \in \R$ $$ m_{\mu}^{\a}(z,x,\oo)\le m_{\mu}^{\a}(y,x,\oo)+m_{\mu}^{\a}(z,y,\oo).$$
		\end{proposition}
		\begin{prop}[{{Homogenization of the metric problem  \cite[Proposition 4.1]{homo.stoch.premier.ordre.quasiconvex}}}\label{existence des solutions de pmlj}]\label{pro:3.3}
		There exists an event $\o_{3}\subseteq \o$ of full probability and a function $\overline{m}_{\mu}^{\a}$ such that, for every $\oo \in \o_{3}$, $y \in \R$ and $\mu \ge \mu^{\star}_{\alpha} $ we have 
			 \begin{equation}\label{conver de l SOLUT DE PMLJ}
			\frac{m_{\mu}^{\a}(ty,0,\oo)}{t} \to \overline{m}_{\mu}^{\a}(y)~~~~\textrm{as } t\to \infty.	
			\end{equation} 
		\end{prop}
		
		We can now define the deterministic hamiltonians $\overline{H}_{\a}$. 
				\begin{definition}[Definition of the effective hamiltonians far from the junction]
		
			The deterministic hamiltonians are given by $$ \overline{H}_{\a} (p)=\inf\limits_{\phi \in S'}\sup \limits_{y\in \R}[H_{\a}(D\phi+p,y,\oo)] =\inf\{\mu;\overline{m}_{\mu}^{\a}\ge p.y; \forall y\}.$$
					It's immediate from this definition and Assumptions ${\bf (H)}$ that $\overline{H}_{\a}$ are continuous, coercive and convex in $p$. 

		\end{definition}
		
		We also have the following results, which will be useful in the rest of the paper
		\begin{lem}[{{Representation formulae for $\overline{m}_{\mu}^{\alpha}$ \cite[Lemma 4.2]{homo.stoch.premier.ordre.quasiconvex}}}]\label{lem:3.5}
			For every $\mu >\mu^{\star}_{\alpha} $, \begin{equation}\label{definition de la solution du probleme metrique loin de la jonction}
			\overline{m}_{\mu}^{\alpha}(y)=\max\{p.y;\overline{H}_{\alpha}(p) \le\mu\}.\end{equation} 
		\end{lem}

		\begin{corollary}[{{New definition of $\mu^{\star}_{\alpha}$ \cite[Corollary 4.3]{homo.stoch.premier.ordre.quasiconvex}}}]\label{cor:3.6} We have, under Assumptions ${\bf (H)}$
			$$\mu^{\star}_{\alpha}=\min\limits_{p} \overline{H}_{\a} (p)=\overline{H}_{\a}(0). $$ 
		\end{corollary}
		\begin{prop}[{{Equation satisfied by  $\overline{m}_{\mu}^{\a}$ \cite[Corollary 4.5]{homo.stoch.premier.ordre.quasiconvex}}}] \label{probleme limite des right and left hamiltonian}
			For all $\mu \ge \mu^{\star}_{\alpha}$, the function $\overline{m}_{\mu}^{\a}$ is the solution of the following problem \begin{equation}\label{pmljlimite}   \left\{
			\begin{aligned}
			&\overline{H}_{\a}(D\overline{m}_{\mu}^{\a}(y))=\mu,  & y \in \R \setminus \{0\}, \\
			&\overline{m}_{\mu}^{\a}(0)=0,
			\end{aligned}
			\right.\end{equation}
		\end{prop}
\begin{prop}[{{Approximated correctors far from the junction \cite[Lemma 5.1]{homo.stoch.premier.ordre.quasiconvex}}}]\label{pro:homo.stoch.premier.ordre.quasiconvex}
For all $\d>0$, $\oo \in \o$ and $p\in \R$, there exists a unique bounded solution $v^{\d}_{\a}(\cdot,\oo,p)$	of $$ \d v^{\d}_{\a}+H_{\a}(Dv^{\d}_{\a}+p,y,\oo)=0 \mbox{ in }\R  .$$ In addition, there exists an event $\o_{2}\subset\o$ of full probability such that for all $\oo \in \o_{2}$ $$ \d v^{\d}_{\a}(y,\oo,p) \to -\overline{H}_{\a}(p)\mbox{ in a ball of radius }1/\d. $$	
		\end{prop}

\begin{rem}
Note that all the results presented in this section hold true for $H_0$.
\end{rem}
\section{Metric problem at the junction point}\label{metric problem sur la jonction}
\subsection{Definition of the metric problem at the junction and of the stochastic flux limiter}\label{sec:41}
This subsection is devoted to the definition of the flux limiter. In this subsection, we define a stochastic flux limiter. We will show in sub-section \ref{sec:43} that, under assumption ${\bf (H7)}$, this flux limiter is in fact deterministic. 
We begin by the definition of  the metric problem at the junction point: for $\oo \in \o$, $\mu \in \R$, we consider
		\begin{equation}\label{P.M.J}\left\{
		\begin{aligned}
		&H\left(Dm_{\mu}(y,0,\oo),y,\oo\right)=\mu,  & y \in \R \setminus \{0\}, \\
		&m_{\mu}(0,0,\oo)=0.
		\end{aligned}
		\right.
		\end{equation}
The main difficulty comes here from the fact that the hamiltonian $H$ is not stationary ergodic. As in the classical case, we start 
by the definition of  the smallest value such that the metric problem admits a sub-solution and we denote it by $\tilde{A}(\oo)$:
 $$ \tilde{A}(\oo)=\inf\{\mu; \mbox{ there exits}~v\in \mathbf{Lip};~ H(Dv,y,\oo) \le \mu~\mbox{in }	~\R\}. $$
Contrary to the classical case, this constant is a priori not deterministic. Nevertheless, using ${\bf (H7)}$, we will prove that it's deterministic.

We begin by a new definition for  $\tilde{A}(\oo)$.
\begin{lem}[New definition of $\tilde{A}(\oo)$]
Assume ${\bf (H1)}$-${\bf (H6)}$, then, for all $\oo\in \o$, the stochastic flux limiter $\tilde{A}(\oo)$ is given by $$\tilde{A}(\oo)=\inf\limits_{v \in \mathbf{Lip}} \sup_{y \in \R}  H(Dv,y,\oo)=\sup_{y \in \R}  H(0,y,\oo)<\infty.$$
\end{lem} 

\begin{proof} Note first that Assumption ${\bf (H6)}$ implies that
$$\inf\limits_{v \in \mathbf{Lip}} \sup_{y \in \R}  H(Dv,y,\oo)=\sup_{y \in \R}  H(0,y,\oo).$$
		 Let $\oo\in  \o $. We first prove that $$\tilde{A}(\oo)\ge \sup\limits_{y \in \R}  H(0,y,\oo). $$ By definition of $\tilde{A}(\oo)$, for $\eta >0$, there exists $\mu_{\eta}\in \R$ and $v_{\eta}\in \mathbf{Lip}$ such that $\mu_{\eta}< \tilde{A}(\oo)+\eta$, and $$H(Dv_{\eta},y,\oo )\le \mu_{\eta}<\tilde{A}(\oo)+\eta \mbox{ for all }y \in \R.$$ Assumption ${\bf (H6)}$, yields that 
		$$  \sup\limits_{y \in \R}  H(0,y,\oo)<\tilde{A}(\oo)+\eta \mbox{ for all }\eta>0 . $$Letting $\eta \to 0$, we obtain the result.   
		
To prove the reverse inequality, we consider $\mu=\sup\limits_{y \in \R}H(0,y,\oo)<\infty$. Then all constant function $v$ are solution of 
$$H(Dv,y,\oo)= H(0,y,\oo)\le \mu,$$
and so, we get $\tilde{A}(\oo)\le \mu$. This implies that $\tilde{A}(\oo)=\sup\limits_{y \in \R}  H(0,y,\oo)<\infty$ and concludes the proof. 
\end{proof}
		
We now give an existence result for the metric problem.
\begin{thm}[Existence of solution for the metric problem]\label{theo.existence.de.la.solution.du.pmj}
Assume ${\bf (H1)}$-${\bf (H6)}$.	For all $\oo \in \o$ and $\mu \ge \tilde{A}(\oo)$, there exists a solution $m_{\mu}(\cdot,0,\oo) \in \mathbf{Lip}$ of \eqref{P.M.J}.

		\end{thm}
		In order to prove Theorem \ref{theo.existence.de.la.solution.du.pmj}, we need  a comparison principle for equation \eqref{P.M.J}. 
	\begin{prop}[Comparison principle for \eqref{P.M.J}]\label{principe de comparaison du probleme metrique sur la jonction}Assume ${\bf (H1)}$-${\bf (H6)}$ and fix  $\oo \in \o$ and $\mu >\tilde{A}(\oo)$. Assume that $u \in USC(\R \setminus{\{0\}})$ and $v \in LSC(\R\setminus\{0\})$ are respectively a sub- and a super-solution of \eqref{P.M.J}. Assume also that $\liminf\limits_{|y|\to +\infty}|y|^{-1}v(y,\oo)> 0$ then 
$$ u(\cdot,\oo) \le v(\cdot,\oo)~\mbox{in}~\R\setminus\{0\}. $$
		\end{prop}
		\begin{rem}
			Note that if $u$ is a solution of $H(Du,y,\oo) \le \mu$ then by coercivity assumption {\bf (H3)}, $u$ is Lipschitz continuous.
		\end{rem}

\begin{proof}
We want to prove a comparison principle on an unbounded domain for unbounded functions. First, since $u$ is a sub-solution of \eqref{P.M.J} then it's Lipschitz continuous and we have 
		$$L= \limsup\limits_{|y|\to \infty }|y|^{-1}u(y)<\infty.$$
 It's enough to prove this result for $L\ge 0$, since otherwise the result is immediate from the usual comparison principle on bounded domains (see \cite{Ishi}). 
 We define 
 $$E=\{ 0\le \lambda\le 1: \liminf\limits_{|y|\to +\infty}\frac{v(y,\oo)-\lambda u(y,\oo)}{|y|}\ge 0  \},$$ 
  and $\gamma =\sup E$. By assumptions, we have that $\gamma >0$ and $[0,\gamma)\subseteq E$. We also have that, for any $0<\eta<\gamma$,  
  $$\liminf\limits_{|y|\to +\infty}\frac{v(y,\oo)-\gamma u(y,\oo)}{|y|}=\liminf\limits_{|y|\to +\infty}\frac{v(y,\oo)-(\gamma-\eta) u(y,\oo)}{|y|}-\eta \limsup\limits_{|y|\to +\infty}\frac{ u(y,\oo)}{|y|}\ge -\eta L.$$
  Hence letting $\eta \to 0$, we get $\gamma \in E$ and $E=[0,\gamma]$. 
  
  We now show that $\gamma=1$. By contradiction, assume that $\gamma <1$. Fix $0<\lambda\le \gamma<1$, $R,\eta>0$ (to be selected below) and $0\le \d <\min(\frac{1}{2}(1-\lambda),\eta/2L) $. We define $\tilde{u}=(\lambda+\delta)u$. Using the convexity of $H$ in $p$, we have
    $$H(D\tilde u,y,\omega)\le (\lambda+\delta)H(D u,y,\omega)+(1- (\lambda+\delta))H(0,y,\omega)\le  (\lambda+\delta)\mu + (1- (\lambda+\delta))\tilde A(\omega)\le \mu-\alpha$$
    for some $\alpha>0$ depending only on $\lambda$. We also define $\tilde{v}=v+\eta ((R^{2}+|y|^{2})^{1/2}-R)=v+\eta \phi_{R}(y).$
 Since $H$ is Lipschitz continuous in $p$, and for a suitable $\eta$, we have that 
 $$ H(D \tilde{v},y,\oo)\ge \mu -\a.$$
Since $\lambda \in E$ and by the choice of $\d $, we deduce that
$$ \liminf\limits_{|y|\to +\infty}\frac{\tilde v(y,\oo)-\tilde{u}(y,\oo)}{|y|}\ge\eta/2>0. $$
Hence, we can apply the comparison principle for bounded domains to $\tilde{u}$ and $\tilde{v}$ in order to get
$$\tilde{u}-\tilde{v}\le 0\mbox{ in }\R \setminus\{0\}.$$
Letting $R \to +\infty$, we deduce that, since $\phi_{R}\to 0$, $$\tilde{u}-v\le 0\mbox{ in }\R \setminus\{0\},$$ 
i.e.
\begin{equation}\label{eq:1} (\lambda+\delta)u-v \le 0. 
\end{equation}
Dividing by $|y|$ and taking the $\limsup$, we get
 $$\limsup\limits_{|y|\to \infty}|y|^{-1} ((\lambda+\d)u-v)\le 0$$ 
 and so $\lambda+\d\in E$. Taking $\lambda=\gamma$, and $\d >0$, we get that $\lambda+\d>\gamma $, which is a contradiction with the definition of $\gamma$. So we obtain that $E=[0,1]$. 
 
 Using the same arguments and taking $\d =0$ and $\lambda \to 1$, we deduce the result using \eqref{eq:1}.
 \end{proof}
				\medskip
		
We now give the proof of Theorem  \ref{theo.existence.de.la.solution.du.pmj}
\begin{proof}[{Proof of Theorem \ref{theo.existence.de.la.solution.du.pmj}}]
Fix $\oo \in \o $ and let $\mu> \tilde{A}(\oo)$. 
We set $$ M=\{v \in \mathbf{Lip}~\mbox{ such that }H(Dv,y,\oo)\le \mu \},$$and for each $y \in \R$, $$ m_{\mu}(y,0,\oo)=\sup\limits_{v \in M}\{v(y,\oo)-v(0,\oo)\}. $$ Since $ \mu > \tilde{A}(\oo)$, we have $M \neq \emptyset$, so the previous supremum is well defined. We now want to show that $ m_{\mu}$ is finite. To do that, we introduce a regular function $\psi$ defined on $\R$, for $C_{\mu}$ large enough, by
		$$ \psi(y) =C_{\mu}|y| .$$
 The coercivity of $H$ yields that $\psi $ is a global super-solution of \eqref{P.M.J}. Then the comparison principle Proposition \ref{principe de comparaison du probleme metrique sur la jonction}, applied to $v(\cdot,\oo)-v(0,\oo)$ and $\psi$ for each $v \in M$, yields that $$v(y,\oo)-v(0,\oo) \le \psi(y),  $$ which implies that $$m_{\mu}(y,0,\oo)\le \psi (y),$$  and so $m_{\mu}$ is finite.  Using that, for all $v\in M$, 
 $$H(0,y,\omega)\le H(Dv,y,\omega),$$
 we deduce that $0\in M$ and so $m_\mu(y,0,\omega)\ge 0.$		
		Moreover, by construction we have that $ m_{\mu}$ is a solution of $$H(Dv,y,\oo)\le \mu \mbox{ for } y \in \R ,  $$ and satisfies $m_{\mu}(0,0,\oo)=0$. In particular $m_{\mu}$ is a sub-solution of \eqref{P.M.J}. The fact that $m_{\mu}$ is a viscosity super-solution follows directly from the definition of $m_{\mu}$ and the following lemma 
		\begin{lem}[{\cite[Lemma 2.1]{dragoni}}]\label{contradiction d une supersolution. oar.la.methode de peroon}
			If $v \in M$ is not a super-solution of \eqref{P.M.J} then there exists $u \in M$ such that $$v(y,\oo) < u(y,\oo) ~\mbox{at some}~y \in \R.   $$
		\end{lem}    
\bigskip
		
It remains to show that there exists a solution for $\mu =\tilde{A}(\oo)$. We have, by definition of $m_{\mu}$, that the map 
$$ \mu \mapsto m_{\mu}(y,0,\oo)  ~\mbox{is non-decreasing for } \mu> \tilde{A}(\oo).$$ 
We then define, for all $y \in \R$, 
$$ m_{\tilde{A}(\oo)}(y,0,\oo)=\inf\limits_{\mu (\oo) > \tilde{A}(\oo)}m_{\mu}(y,0,\oo)=\lim\limits_{\mu (\oo) \to  \tilde{A}(\oo)}m_{\mu}(y,0,\oo). $$
By stability, we deduce that  $ m_{\tilde{A}(\oo)}(\cdot,0,\oo) $ is a solution of \eqref{P.M.J}.
\end{proof}

\subsection{Properties of the solution of the metric problem}
For all $\oo\in \o$, $\mu \ge \tilde{A}(\oo)$ and $x \in \R$,  we consider 
		\begin{equation}\label{P.M.J.for.all.x}\left\{
		\begin{aligned}
		&H\left(Dm_{\mu}(y,x,\oo),y,\oo\right)=\mu , & y \in \R \setminus \{x\}, \\
		&m_{\mu}(x,x,\oo)=0.
		\end{aligned}
		\right.
		\end{equation}
		
As in Section \ref{sec:41}, we can prove that there exists a solution $m_{\mu}(\cdot,x,\oo)$ of \eqref{P.M.J.for.all.x} given by the following representation formulae
		\begin{equation}\label{repres.de.la.solu.de.pmj.pour.tout.x}
		m_{\mu}(y,x,\oo)=\sup\{v(y,\oo)-v(x,\oo);v\in \mathbf{Lip}\mbox{ such that }H(Dv,y,\oo)\le \mu \mbox{ in }\R\}.
		\end{equation}
		
\begin{rem}
In the case of stationary ergodic hamiltonian, we have $m_\mu(y,x,\tau_{-x}\oo)=m_\mu(y-x,0,\oo)$ (see \cite{homog.stoch.quantitative.premier.ordre.concexe} for example). In our case, this property is no longer true.
\end{rem}

We can extend this definition from $\{x\}$ to any arbitrary compact $K\subseteq\R$ and we define 
		$$ m_{\mu}(y,K,\oo)=\inf\limits_{x \in K}\sup\{v(y,\oo)-v(x,\oo);v\in \mathbf{Lip}\mbox{ such that }H(Dv,y,\oo)\le \mu \mbox{ in }\R\}. $$
		The function $m_{\mu}(\cdot,K,\oo)$ is then a solution of the following problem
		\begin{equation}\label{P.M.J.for.k}\left\{
		\begin{aligned}
		&H\left(Dm_{\mu}(\cdot,K,\oo),y,\oo\right)=\mu  & y \in \R \setminus K, \\
		&m_{\mu}(\cdot,K,\oo)=0& \mbox{ in }K.
		\end{aligned}
		\right.
		\end{equation}

\begin{proposition}[{\cite[Proof of Proposition 3.6]{homog.stoch.quantitative.premier.ordre.concexe}}]\label{representation de m mu avec k}
For every $\oo\in \o$ and $\mu>\tilde A(\oo)$, the function $m_{\mu}(\cdot,K,\oo) $ is the unique nonnegative solution of \eqref{P.M.J.for.k} and is also 
  given by the following representation formulae, for all $y \in \R \setminus K$
		\begin{equation}\label{nouvelle representation de la solution du pmj avec k}m_{\mu}(y,K,\oo)=\inf\limits_{x \in K}\sup\{v(y,\oo)-v(x,\oo);v\in \mathbf{Lip}\mbox{ such that }H(Dv,y,\oo)\le \mu \mbox{ in }\R\setminus K\}.\end{equation}			

		\end{proposition}
A direct consequence of the previous proposition is the following corollary.
		\begin{corollary}
			The function $m_{\mu}(\cdot,K,\oo)$ is $\textbf{F}(\R\setminus K)$-measurable.
		\end{corollary}

		The functions $m_{\mu}(\cdot,x,\oo)$ and $m_{\mu}(\cdot,K,\oo)$ satisfy several properties that are recalled in the following proposition.
		\begin{prop}[{{Some properties of $m_\mu$}}]\label{proprietes verifie par la soluton du probleme metrique}
			For all $\oo\in  \o$ and $\mu > \tilde {A}(\oo)$, the functions $m_{\mu}(\cdot,x,\oo)$ and $m_{\mu}(\cdot,K,\oo)$ verify the following properties
			\begin{itemize}
				\item If $U\subseteq \R$ is open, $x\in \R \setminus U $ and $u$ verifies $$ H(Du,y,\oo)\le\mu \mbox{ in } U,$$ then \begin{equation}\label{sous solution sur un un  ouvert u avec la solution du p.m.j propritete de maximalite}
				u-m_{\mu}(\cdot,x,\oo) \le \max\limits_{y\in  \partial U }\left(u(y)-m_{\mu}(y,x,\oo)\right) \mbox{ in } U.	\end{equation}  
				\item For all $y, z \in \R$ we have $$ m_{\mu}(y,x,\oo) \le m_{\mu}(y,z,\oo)+m_{\mu}(z,x,\oo) ;$$
				\item $m_{\mu}(\cdot,x,\oo)$ is equivalent to the euclidean metric, i.e, for $\oo \in \o$; there exists $l_{\mu}, L_{\mu}\ge 0$ satisfying, for some $C',c' >0$, $c'(\mu - \tilde A(\oo))\le l_{\mu}\le L_{\mu}\le C'$ such that  \begin{equation}\label{ineg.sol.metr.prob}0\le l_{\mu} |y-x| \le m_{\mu}(y,x,\oo)\le L_{\mu}|y-x| ;\end{equation} 
				\item For all $z \in \R \setminus K$ and $y \in \R \setminus K'$ where $K,K'$ two compacts in $\R$, we have
				\begin{equation}\label{inegalite.lipsh.compacts}
				|m_{\mu}(z, K,\oo)-m_{\mu}(y, K',\oo)| \le L_{\mu}(|z-y|+dist_{\mathcal H}(K,K')), 
				\end{equation} with $dist_{\mathcal H}$ is the Hausdorff distance given by $$ dist_{\mathcal{H}}(E,F)=\max\left\{\sup\limits_{y \in F}\inf\limits_{x \in E}|y-x|,\sup\limits_{y \in E}\inf\limits_{x \in F}|y-x|\right\} .$$
				\item For every open set $U\subseteq \R$, $x\in U$ and $y \in \R\setminus U$ \begin{equation}\label{donner la solution du probleme metrique sous formr de minimum et de somme sur le bord et la valeur de m en cette valeur}
				m_{\mu}(y,x,\oo)=\min\limits_{ z \in \partial U}\left(m_{\mu}(y,z,\oo)+m_{\mu}(z,x,\oo)\right).
				\end{equation}
			\end{itemize}
		\end{prop}
\begin{proof}
The proofs of these properties are based on the maximality of $m_{\mu}$, assumptions ${\bf (H2)}$, ${\bf (H3)}$ and the comparison principle Proposition \ref{principe de comparaison du probleme metrique sur la jonction}. In particular, these properties don't use the stationarity of $H$. For the reader's convenience, we just prove inequality \eqref{ineg.sol.metr.prob}, the other proofs being similar to the ones in \cite[Proposition 3.1]{homog.stoch.quantitative.premier.ordre.concexe}. 

Since the hamiltonian $H$ is coercive, convex and verifies ${\bf (H6)}$, then for all $\mu > \tilde {A}(\oo)$ there exists $p^{+}_{y}>0$ and $p^{-}_{y}<0$ such that $H(p^{+}_{y},y,\oo)=H(p^{-}_{y},y,\oo)=\mu$. Moreover, for all $p \in [p^{-}_{y},p^{+}_{y}]$, we have that $H(p,y,\oo)\le \mu$ and 
$$\mu- \tilde {A}(\oo)\le H(p_y^\pm,y,\oo)-H(0,y,\oo) \le c|p_y^\pm| .
$$
 Let $l_{\mu}=\frac{1}{c}(\mu- \tilde {A}(\oo))$. Then we have that $p^{-}_{y}\le -l_{\mu}\le l_{\mu} (y-x)/|y-x|\le l_{\mu}\le p^{+}_{y}$. 
Hence, the function $y \to l_{\mu}|y-x|$ satisfies
$$   H(l_{\mu} (y-x)/|y-x|,y,\oo)\le \mu  .  
$$ 
So, by the maximality's property of $ m_{\mu}(\cdot,x,\oo) $, we obtain the left side of \eqref{ineg.sol.metr.prob}. For the second inequality,  we have that the function $y \to L_{\mu}|y-x|$ is, by coercivity of $ H$,  a super-solution of \eqref{P.M.J.for.all.x}, for $L_{\mu}$ large enough. Then we obtain the right side of \eqref{ineg.sol.metr.prob} by comparison. 
\end{proof}

		We now give some consequences of the properties given in Proposition \ref{proprietes verifie par la soluton du probleme metrique}. In particular, we will give a localization property of the solution $m_{\mu}(\cdot,x,w)$ of \eqref{P.M.J.for.all.x}. We begin by the definition of the sub-level set of $m_{\mu}(\cdot,x,\oo)$.
		\begin{definition}
			For each $t >0$, $\mu$, we define the reachable set to $x$ in time $t$ by
			$$\mathcal R_{\mu,t}(x)=\mathcal R_{\mu,t}=\{(y,\oo)\in \R\times \o; m_{\mu}(y,x,\oo) \le t\},$$ 
and for $\oo\in \Omega$ and $\mu>\tilde A(\oo)$,
$$
\mathcal R_{\mu,t}^{\oo}=\{y \in \R ; (y,\oo) \in \mathcal R_{\mu,t}\}.$$  
		\end{definition} 
		If we consider, for $x\in \R$, $U$ as $$ U=\{y \in \R; m_{\mu}(y,x,\oo)<t\},$$then \eqref{donner la solution du probleme metrique sous formr de minimum et de somme sur le bord et la valeur de m en cette valeur} and Proposition \ref{representation de m mu avec k}  yields that, for every $t>0$ and $y\in \R$ such that $ m_{\mu}(y,x,\oo)\ge t$, \begin{equation}\label{la solution du pmj comme somme de t plus min dur le sous niveau}
		m_{\mu}(y,x,\oo) =t+ \min\limits_{z \in \mathcal R_{\mu,t}^{\oo}}m_{\mu}(y,z,\oo)=t+m_{\mu}(y,\mathcal{R}_{\mu,t}^{\oo},\oo).\end{equation}
		Now using \eqref{ineg.sol.metr.prob}, we obtain that, for every $t$ and $\oo$,\begin{equation}\label{l inclusion de sous niveau utilisant l equivalence a la distance euclidenne}
		B _{t/L_{\mu}}\subseteq\mathcal R_{\mu,t}^{\oo}\subseteq B _{t/l_{\mu}},
		\end{equation} where $B _{r}$, is a ball of radius $r$  centred in  $x$. Then, $\mathcal R_{\mu,t}^{\oo}$ is a compact set in $\R$.
We now give the property of localization of $m_{\mu}(\cdot,x,\oo)$.
		\begin{lem} [{{\cite[Lemma 3.4]{homog.stoch.quantitative.premier.ordre.concexe}}}]\label{localistion de la solution du pmj}
			For every $\oo \in \o,~\mu>\tilde{A}(\oo)$, $x\in \R$, and $u\in \textbf{Lip}(R_{\mu,t}^{\oo})$, $$H(Du,y,\oo)\le \mu \mbox{ in } \mathcal R_{\mu,t}^{\oo} \mbox{ implies that }u(\cdot,\oo)-u(x,\oo) \le m_{\mu}(\cdot,x,\oo) \mbox{ in } \mathcal R_{\mu,t}^{\oo}.
$$
		Then we deduce that we can represent the solution $m_{\mu}(\cdot,x,\oo)$ of \eqref{P.M.J.for.all.x} in $\mathcal{ R}_{\mu,t}^{\oo}$ as 
		$$ m_{\mu}(y,x,\oo)=\sup\{w(y,\oo)-w(x,\oo); w\in \mathbf{Lip}(\mathcal{R}_{\mu,t}^{\oo}) ,H(Dw,y,\oo)\le \mu \mbox{ in }\mathcal{R}_{\mu,t}^{\oo}       \} .  $$		\end{lem}
		\begin{lem}
		For every $\oo\in \o$, $\mu >\tilde{A}(\oo)$ and $0\le t\le s$, we have that 
		\begin{equation}\label{distance entre les sous level}
		dist_{\mathcal{H}}(\mathcal{R}_{\mu,t}^\oo,\mathcal{R}_{\mu,s}^\oo)\le \frac{1}{l_\mu}(s-t).
		\end{equation}
		\end{lem}
		\begin{proof}
		We have, for $t \le s$, that $\mathcal{R}_{\mu,t}^{\oo}\subseteq \mathcal{R}_{\mu,s}^{\oo} $. So to obtain the result it's enough to prove that 
$$\mathcal{R}_{\mu,s}^{\oo}\subseteq \mathcal{R}_{\mu,t}^{\oo}+ \frac{1}{l_\mu}(s-t)\overline{{B}}_{1}.$$
By  \eqref{ineg.sol.metr.prob} and \eqref{la solution du pmj comme somme de t plus min dur le sous niveau}, we have, for all $y$ such that $m_{\mu}(y,x,t)\ge t$, that
$$  
t+l_{\mu}dist(y,\mathcal{R}_{\mu,t}^\oo)\le t+m_{\mu}(y,\mathcal{R}_{\mu,t}^\oo,\oo)=m_{\mu}(y,x,\oo). 
$$In particular, for $y \in \mathcal{R}_{\mu,s}^{\oo}$, we have that 
$$ 	t+l_{\mu}dist(y,\mathcal{R}_{\mu,t}^{\oo})\le m_{\mu}(y,x,\oo)\le s. 
$$
Thus for all $y\in  \mathcal{R}_{\mu,s}^\oo\setminus \mathcal{R}_{\mu,t}^\oo $, we have
$$dist(y,\mathcal{R}_{\mu,t}^{\oo})\le\frac {s-t}{l_\mu}.$$
This completes the proof. 
		\end{proof}

\subsection{Deterministic flux limiter}\label{sec:43}
In this subsection, we show that under Assumption ${\bf(H7)}$ (either ${\bf (H7}$-${\bf WFL)}$ or ${\bf (H7}$-${\bf \eps)}$), the stochastic flux limiter is in fact deterministic.
\subsubsection{The case of Assumption ${\bf (H7}$-${\bf WFL)}$}
 
\begin{thm}[Definition of $\overline A$]\label{le.flux.limitteur.est.deterministe}
 Assume ${\bf (H}$-${\bf WFL)}$. Then, there exists $\o_{1}$ of full probability such that, for all $\oo\in \o_{1}$, the stochastic flux limiter $\tilde A(\omega)$ is given by
 $$\tilde{A}(\oo) =\overline A:= \max (\mu ^{\star}_{L},\mu ^{\star}_{R}).$$
  \end{thm}

 \begin{proof}
 We begin to prove that
\begin{equation} \label{le flux limitteur est deterministe d en bas}
	\tilde{A}(\oo) \ge\max (\mu ^{\star}_{L},\mu ^{\star}_{R}).
\end{equation}
 To do that we need the following lemma
		\begin{lem}
There exists $\o_{1}$ of full probability such that for all $\oo \in \o_{1}$, we have
	$$\sup_{y\in[1,+\infty[} H_R(0,y,\oo)=\sup_{y\in \R} H_R(0,y,\oo)\quad {\rm and}\quad \sup_{y\in]-\infty, -1]} H_L(0,y,\oo)=\sup_{y\in \R} H_L(0,y,\oo).$$
		\end{lem}
\begin{proof}
  We have to prove that $$P\left(	\sup_{y\in[1,+\infty[} H_R(0,y,\oo)\neq \sup_{y\in \R} H_R(0,y,\oo)\right) =0.$$
  To simplify the presentation, we assume that $\sup\limits_{y\in[2k,2k+1[} H_R(0,y,\oo)$ and $\sup\limits_{y\in[2k+1,2k+2[} H_R(0,y,\oo)$ are $\textbf{F}$-measurable, otherwise, we have to approximate the sup by a sup on a countable set. By assumption ${\bf (H9)}$, we have that either   
 $$P\left(\sup\limits_{y\in[2k,2k+1[} H_R(0,y,\oo)< \mu^{\star}_{R}\right)<1\quad{\rm or}\quad P\left(\sup\limits_{y\in[2k+1,2k+2[} H_R(0,y,\oo)< \mu^{\star}_{R}\right)<1.$$
 We assume that $P\left(\sup\limits_{y\in[2k,2k+1[} H_R(0,y,\oo)< \mu^{\star}_{R}\right)<1$, the other case being similar. By assumption \textbf{(H9)}  we then have 
 $$
    \begin{aligned}
    &P\left(	\sup_{y\in[1,+\infty[} H_R(0,y,\oo)\neq \sup_{y\in \R} H_R(0,y,\oo)\right)
    \\&= P\left(\cap_{n\ge 1}\left\{	\sup_{y\in[n,n+1[} H_R(0,y,\oo)< \sup_{y\in \R} H_R(0,y,\oo)\right\}\right)
 \\&\le P\left(\cap_{k\ge 1}\left\{	\sup_{y\in[2k,2k+1[} H_R(0,y,\oo)< \sup_{y\in \R} H_R(0,y,\oo)\right\}\right)\\&= \lim\limits_{N\to \infty} \prod_{k=1}^{N}P\left(\sup_{y\in[2k,2k+1[} H_R(0,y,\oo)< \mu^{\star}_{R}
 \right)\\&=\lim\limits_{N\to \infty}\left(P\left(\sup_{y\in[2k,2k+1[} H_R(0,y,\oo)< \mu^{\star}_{R}
 \right)\right)^{N}= 0.
 \end{aligned} 
 $$
\end{proof}
To prove that $\tilde{A}(\oo) \ge \max(\mu^{\star}_{L},\mu^{\star}_{R})$, we assume, by contradiction, that $\tilde A(\oo)<\mu^\star_R$ (the case $\tilde A(\oo)<\mu^\star_L$ being similar).
	For all $\oo\in \o_{1}$, we then have
	$$\sup_{y\in \R} H_R(0,y,\oo)>\tilde A(\oo)=\sup_{y\in \R} H(0,y,\oo)\ge \sup_{y\in[1,+\infty[} H(0,y,\oo)= \sup_{y\in[1,+\infty[} H_R(0,y,\oo),$$
	which is absurd.

	We now prove that 
$$\tilde{A}(\oo)\le \max (\mu_{L}^{\star},\mu_{R}^{\star}).$$ 
		
 Under assumptions ${\bf (H}$-${\bf WFL)}$, we have that 
 \begin{align*}
& \{\mu; \mbox{ s.t } \exists v \in \mathbf{Lip}; \max(H_{L}(Dv,y,\oo),H_{R}(Dv,y,\oo)) \le \mu~\mbox{in }	~\R\} \\
 \subset& \{\mu; \mbox{ s.t } \exists v \in \mathbf{Lip}; H(Dv,y,\oo) \le \mu~\mbox{in }	~\R\} 
 \end{align*}
  then we get 
  $$ \max(\mu_{L}^{\star},\mu_{R}^{\star})\ge \tilde{A}(\oo). $$
We conclude that $\tilde{A}$ is deterministic in $  \o_{1}$.
		
\end{proof}

\subsubsection{The case of Assumption ${\bf (H7}$-${\bf \eps)}$}
As in Proposition \ref{pro:3.1}, we denote $\mu_0^\star$ by
$$
\mu^{\star}_{0} = \inf\{\mu;\mbox{ s.t there exists a function}~v \in \mathbf{Lip} ~\mbox{which satisfies}~H_{0}(Dv,y,\oo) \le \mu \mbox{ in } \R\}. 
$$
As in Corollary \ref{cor:3.6}, $\mu_0^\star=\inf_p \overline H_0(p)$.
\begin{thm}[Definition of $\overline A$]\label{le.flux.limitteur.est.deterministe-eps}
 Assume ${\bf (H}$-${\bf \eps)}$. Then, there exists a decreasing sequence of set $\o_1^{\eps}$ such that $P(\o_1^\eps)=1-\theta(\eps)$ with $\theta (\eps)\to 0$ as $\eps \to 0$ and such that, for all $\oo\in \o_{1}^\eps$, the stochastic flux limiter $\tilde A(\omega)$ is given by
 $$\tilde{A}(\oo) =\overline A:= \max(\min\limits_{p}\overline{H}_{0}(p),\min\limits_{p}\overline{H}_{L}(p),\min\limits_{p}\overline{H}_{R}(p))=\min\limits_{p}\overline{H}_{0}(p)=\mu^{\star}_{0}.$$
  \end{thm}
 
\begin{proof}
We begin to prove that 
\begin{equation}\label{eq:100}
\tilde{A}(\oo) \ge \mu^{\star}_{0}.
\end{equation}
 To do that we need the following lemma
		\begin{lem}
There exists a decreasing sequence of set $\o_1^{\eps}$ such that $P(\o_1^\eps)=1-\theta(\eps)$ with $\theta (\eps)\to 0$ as $\eps \to 0$ and such that, for all $\oo\in \o_{1}^\eps$, we have
$$\sup_{y\in[\frac{-1}{\sqrt{\eps}},\frac{1}{\sqrt{\eps}}]} H_0(0,y,\oo)=\sup_{y\in \R} H_0(0,y,\oo).
$$
\end{lem}

\begin{proof}
  We have to prove that
  $$P\left(	\sup_{y\in[\frac{-1}{\sqrt{\eps}},\frac{1}{\sqrt{\eps}}]} H_0(0,y,\oo)\neq
  \sup_{y\in \R} H_0(0,y,\oo)\right) =\theta(\e)\to 0 \; {\rm as}\; \e\to 0.
  $$
  To simplify the presentation, we assume that $\sup\limits_{y\in[2k,2k+1[} H_0(0,y,\oo)$ and $\sup\limits_{y\in[2k+1,2k+2[} H_0(0,y,\oo)$ are $\textbf{F}$-measurable, otherwise, we have to approximate the sup by a sup on a countable set. By assumption ${\bf (H9)}$, we have that either   
 $$P\left(\sup\limits_{y\in[2k,2k+1[} H_0(0,y,\oo)< \mu^{\star}_{0}\right)<1\quad{\rm or}\quad P\left(\sup\limits_{y\in[2k+1,2k+2[} H_0(0,y,\oo)< \mu^{\star}_{0}\right)<1.$$
 We assume that $P\left(\sup\limits_{y\in[2k,2k+1[} H_0(0,y,\oo)< \mu^{\star}_{0}\right)<1$, the other case being similar. By assumption \textbf{(H9)}  we then have 
 $$
    \begin{aligned}
    &P\left(	\sup_{y\in[\frac{-1}{\sqrt{\eps}},\frac{1}{\sqrt{\eps}}]} H_0(0,y,\oo)< \mu^{\star}_{0}\right)
    \\&\displaystyle{\le P\left(\bigcap_{n\in\left\{\left\lfloor\frac{-1}{\sqrt{\eps}}\right\rfloor+1,\dots,\left\lfloor \frac{1}{\sqrt{\eps}}\right\rfloor-1\right\}}\left\{	\sup_{y\in[n,n+1[} H_0(0,y,\oo)<\mu^{\star}_{0}\right\}\right)}
 \\&\displaystyle{\le P\left(\bigcap_{k\in \left\{\frac 12\left\lfloor\frac{-1}{\sqrt{\eps}}\right\rfloor+\frac 12,\dots,\frac 12\left\lfloor \frac{1}{\sqrt{\eps}}\right\rfloor-\frac 12\right\}}\left\{	\sup_{y\in[2k,2k+1]} H_0(0,y,\oo)< \mu^{\star}_{0}\right\}\right)}\\
 &= \prod_{k=\frac 12\left\lfloor\frac{-1}{\sqrt{\eps}}\right\rfloor+\frac 12}^{\frac 12\left\lfloor \frac{1}{\sqrt{\eps}}\right\rfloor-\frac 12}P\left(\sup_{y\in[2k,2k+1[} H_0(0,y,\oo)< \mu^{\star}_{0}
 \right)\\&=\left(P\left(\sup_{y\in[2k,2k+1[} H_0(0,y,\oo)< \mu^{\star}_{0}
 \right)\right)^{ \frac 12\left\lfloor\frac{1}{\sqrt{\eps}}\right\rfloor-\frac 12\left\lfloor\frac{-1}{\sqrt{\eps}}\right\rfloor}=:\theta(\eps)\to0\;{\rm as}\; \e\to 0.
 \end{aligned} 
 $$ 
  
\end{proof}
To prove \eqref{eq:100}, we assume by contradiction that $\tilde A(\oo)<\mu^\star_0$ for all $\oo\in \o_1^\e$. We then have
$$\mu^\star_0=\sup_{y\in \R} H_0(0,y,\oo)>\tilde A(\omega)=\sup_{y\in \R} H(0,y,\oo)\ge \sup_{y\in \left[\frac{-1}{\sqrt{\e}},\frac 1 {\sqrt \e}\right]}H(0,y,\oo)= \sup_{y\in \left[\frac{-1}{\sqrt{\e}},\frac 1 {\sqrt \e}\right]} H_0(0,y,\oo),$$
which is absurd.

	We now prove that 
$$\tilde{A}(\oo)\le \mu_{0}^{\star}.$$ 
		
 Under assumptions ${\bf (H}$-${\bf \e)}$ (and using that $H\le H_0$), we have that 
 \begin{align*}
& \{\mu; \mbox{ s.t } \exists v \in \mathbf{Lip}; H_{0}(Dv,y,\oo)\le \mu~\mbox{in }	~\R\} \\
 \subset& \{\mu; \mbox{ s.t } \exists v \in \mathbf{Lip}; H(Dv,y,\oo) \le \mu~\mbox{in }	~\R\}. 
 \end{align*}
We then get 
  $$ \mu_{0}^{\star}\ge \tilde{A}(\oo) $$
and we conclude that $\tilde{A}$ is deterministic in $  \o_{1}^\e$.

\end{proof}

\subsubsection{A counter example}\label{sec:42}
In this subsection we present a counter-example showing that if the perturbation zone in ${\bf (H7}$-${\bf \e)}$ doesn't depend on $\e$, then, in general,   we don't have the convergence result of $u^\eps$ to a deterministic function. More precisely we cannot obtain a deterministic flux limiter and  the supremum $\sup\limits_{y \in \R}H(0,y,\oo)$ depends on $\oo$ since $H$ is not stationnary. In particular, the limit of $u^\eps$ (if it exists, which is not clear in that case) can't be deterministic. To show that, we  give a counter example for a Bernoulli process. 
		
Given a probability space $(\o,\textbf F,P)$, we define a Bernoulli process $\{X_i\}_{i\in \mathbb N}$ such that, for $\oo\in \Omega$
$$
X_i(\oo)=\left\{
\begin{array}{ll}
0 &\mbox{ with probability }p ,\\
1 &\mbox{ with probability }1-p.
\end{array}
\right.
$$
We then define a   $C^{\infty}(\R)$ function $\overline \psi$ such that
	$$\overline{\psi}(y)=\left\{
		\begin{aligned}
		    &\psi_{0} &|y|\le 1/2,\\
		    & 0 & |y|\ge 3/4,
		\end{aligned}
		\right.$$ 	
where $-1<\psi _{0}<0$ is a constant. We now define a stationary ergodic function $\psi$ by
$$ \psi(y,\oo)=\overline \psi(y-2k) X_k(\oo)+1 \quad {\rm if}\; y \in [2k-1,2k+1].$$
and a stationary ergodic hamiltonian
	$$
		H_{L}(p,y,\oo  )=\left\{
		\begin{aligned}
		&(-p-\tilde{p}-k_{0})\psi(\oo,y), &p<-k_{0}-\tilde{p},\\
		&-|p+\tilde{p}|V_{L}(-\frac{1}{p+\tilde{p}})\psi(\oo,y), & -k_{0}-\tilde{p}\le p\le-\tilde{p},\\
		&(p+\tilde{p})\psi(\oo,y) , &p>-\tilde{p}.
		\end{aligned}
		\right.
		$$
	To define our Hamilton-Jacobi equation, we suppose that the right and left hamiltonians are equal, i.e $H_{L}=H_{R}$ and we consider $H:\R\times\R\times \o\to \R$, given by 
$$H(p,y,\oo)=H_{L}(p,y,\oo)\overline{\psi}(y).$$ 
Hence $H_0=H_L\psi_0$.
The main difference with Assumptions ${\bf (H}$-${\bf WFL)}$ is that the hamiltonian is reduced near the origin (by the factor $\overline \psi$) and so 
$H$ doesn't satisfies $\min(H_L,H_R)\le H\le \max(H_L,H_R)$. Moreover, contrary to the case of Assumption ${\bf (H}$-${\bf \e)}$, the radius of the perturbation's zone is fixed. In that case, which have been treated in \cite{Galise} in the periodic setting, we expect to have  $\bar A>\max(\mu_L^*,\mu_R^*)$, which means that the flux is limited at the origin. 

In the stochastic setting, we can define two sets 
 $$\tilde \o_{0}=\{\oo \in\o; X_0(\oo)=0 \} \mbox{ and }\tilde \o_{1}=\{\oo \in\o; X_0(\oo)=1 \}$$
 with $P(\tilde \o_{0})=p$ and $P(\tilde \o_{1})=1-p$. For $\oo\in \tilde \o_0$, we have that
$$\tilde A(\oo)= \sup\limits_{y \in \R}\left\{\overline{\psi}(y)H_{L}(0,y,\oo)\right\}
=-|\tilde{p}|V_{L}\left(-\frac{1}{\tilde{p}}\right)\psi_{0},$$ 
and for $\oo\in \tilde \o_{1}$ we have that	
$$ \tilde{A}(\oo)=\sup\limits_{y \in \R}\left\{\overline{\psi}(y)H_{L}(0,y,\oo)\right\}=-|\tilde{p}|V_{L}\left(-\frac{1}{\tilde{p}}\right)(\psi_{0}^{2}+\psi_{0}).$$
We then deduce that the flux limiter $\tilde A$ is stochastic and that the limit of $u^\eps$, if it exists, can't be deterministic.	

		\section{Homogenization result of metric problems defined in half space}\label{resultat d homogenization du probleme metrique}
		To prove the homogenization result for \eqref{P.en.epsilon}, it is necessary to prove a homogenization result for the metric problem \eqref{P.M.J.for.all.x}. This result is given in the following theorem.
\begin{thm}[Homogenization of the metric problem \eqref{P.M.J.for.all.x}]\label{homo.du.pmj}Assume ${\bf (H}$-${\bf \e)}$ (resp. ${\bf (H}$-${\bf WFL)})$ and let $\mu>\overline A$. Then, there exists a decreasing sequence of set $\o_0^{\eps}\subset \o_1^{\eps}$ such that $P(\o_0^\eps)=1-\theta(\eps)$ with $\theta (\eps)\to 0$ as $\eps \to 0$ (resp. $\o_0$ of full probability) and a function $\overline{m}_{\mu}:\R\times\R\to [0,+\infty[$  such that, for every $\oo\in \o_{0}^\eps$ (resp. $\oo \in \o_{0}$), we have  
$$
\frac{m_{\mu}(ty,tx,\oo)}{t} \to \overline{m}_{\mu}(y,x)  \mbox{ as }t\to \infty,
$$ 
where $m_{\mu}(\cdot,x,\oo) $ is the solution of \eqref{P.M.J.for.all.x}.
\end{thm} 

When the hamiltonian is stationary ergodic, it is well known that $\overline m_\mu(y,x)$ is in fact a function of $y-x$. This result is no longer true in our setting. Nevertheless, if $x\ge0,y>0$ or $x\le 0,y<0$, we recover this property. This is explained in the following theorem
\begin{thm}\label{homo.du.pmj-bis} Under the same assumptions as the previous theorem, if moreover $x\ge 0,y>0$ (resp. $x\le 0,y<0$) then there exists $\overline{m}_{\mu}^{R}:\R\to [0,+\infty)$ (resp. $\overline{m}_{\mu}^{L}:\R\to [0,+\infty)$) such that
$$\overline{m}_{\mu}(y,x)=\overline{m}_{\mu}^{R}(y-x)\; (\mbox{resp. }=\overline{m}_{\mu}^{L}(y-x)).$$ 	
More precisely, if we define $\overline p_\mu^{\a,\pm}$ for $\a=R,L$ such that $\pm \overline p^{\a,\pm}_\mu\ge0$ and
$$H_\a( \overline p^{\a,\pm}_\mu)=H_\a^\pm( \overline p^{\a,\pm}_\mu)=\mu,$$
then
$$\overline{m}^{R}_{\mu}(y,x)= 
\left\{
\begin{aligned}
&\overline{p}_{\mu}^{R,+}(y-x) &\mbox{ for } y>x;\\
&\overline{p}_{\mu}^{R,-}(y-x) &\mbox{ for } y<x.
\end{aligned}
\right.
$$
$$\left(\textrm{resp. } \overline{m}^{L}_{\mu}(y,x)=
\left\{
\begin{aligned}
&\overline{p}_{\mu}^{L,+}(y-x) &\mbox{ for } y>x;\\
&\overline{p}_{\mu}^{L,-}(y-x) &\mbox{ for } y<x.
\end{aligned}
\right.\right)
$$
	
\end{thm}
	In the rest of this section, we will only treat the case of assumption ${\bf (H}$-${\bf \e)}$, the case of assumption ${\bf (H}$-${\bf WFL)}$ being similar and even simpler.
	
The hamiltonian $H$ is not stationary and then $m_{\mu}$ and $Dm_{\mu}$ are not stationary. In particular, we can't apply the sub-additive ergodic theorem to prove the homogenization. To get the result, we will use the ideas introduced by Armstrong, Cardaliaguet and Souganidis \cite{homog.stoch.quantitative.premier.ordre.concexe}. The proof is decomposed into two steps. In a first time, we prove that, almost surely in $\oo$, the random process    $\frac{1}{t} m_{\mu}(ty,tx,\oo)$ is near to his mean for large $t$. Then, we prove that this mean converges to a quantity, that we denoted by $\overline{m}_{\mu}$. The first part of the proof is given by the following theorem.

		\begin{thm}[Estimate on the fluctuation of $m_{\mu}$] \label{fluctuations}Assume ${\bf (H}$-${\bf \e)}$ and let $\e>0$, $\mu>\overline{A}$ and $m_{\mu}$ be the solution of \eqref{P.M.J.for.all.x}. Then, there exists $C>0$ such that for all $\lambda>0$, and $y, x\in \R$ such that $|y-x|>1$, we have 
			\begin{equation}\label{estimation des fluctuations}
			P(\{\oo\in \o_1^\e, \; |m_{\mu}(y,x,\oo)-E(m_{\mu}(y,x,\oo))|>\lambda\} ) \le exp\left(\frac{-(\mu-\overline{A}) \lambda^{2}}{C|y-x|}\right).\end{equation}
		\end{thm}
Here and below, the expectation $E$ is taken over $\oo\in \o_1^\e$.
		To prove this result we use an argument inspired by the pioneering work of Kesten \cite{kesten} in the theory of first-passage percolation, who introduced a martingale method based on Azuma's concentration inequality. This argument is used in several works and we adapt the one developed in \cite{homog.stoch.quantitative.premier.ordre.concexe}. We also use the ideas developed by Alexander about the level sets \cite{Alexander}.
		The Azuma's inequality is given in the following proposition.
		\begin{prop}\label{proposition de azuma}
			Let $\{X_{k}\}_{k\in \mathbb{N}}$ be a discrete martingale with $X_{0}=0$. Assume that there exists a constant $B>0$, such that, for each $k \in \mathbb{N}$, $$ \sup_{\o}|X_{k+1}-X_{k}|\le B. $$ Then, for each $\lambda>0$ and $N \ge 1 $, $$ P(|X_{N}|>\lambda)\le exp\left( \frac{-\lambda^{2}}{2B^{2}N}\right).  $$
		\end{prop}     
		
		\subsection{A discritization scheme}To obtain the inequality given in Theorem \ref{fluctuations}, we want to define a specific martingale and then use  Azuma's inequality. To do that, we use the localization property of $m_{\mu}$. For this reason, we start by introducing a discretization scheme, where we apply the result for 
		$$ \mathcal R_{\mu,t}^{\oo}=\{y\in \R; m_{\mu}(y,x,\oo)\le t    \}$$
 in order to use the independence hypothesis ${\bf (H9)}$.
		
		We define, $\forall r>0$, 
		$$\mathcal{K}_{r}=\{A \in \mathbf{B}; A=\overline{A} \subseteq [x-r,x+r] \}$$
		the set of compact in $\R$. In particular, $(\mathcal{K}_{r},dist_{\mathcal H})$, where  $dist_{\mathcal H}$ is the Hausdorff distance, is  compact (see Munkres \cite{Munkres}). We fix a small parameter $\delta >0$. Then, there exists $l \in \mathbb{N}$ (depending only on $\delta$ and $r$) and a partition $\Gamma_{1},\ldots , \Gamma _{l}\subset \mathcal {K}_r$ of $\mathcal{K}_{r}$ into Borel subsets, such that 
		$$diam_{\mathcal H}(\Gamma_{i}) \le \d.$$
		Let $K_{m}=\overline{\bigcup\limits_{k \in \Gamma _{m}} k}$. Then $\forall  A \in \mathcal{K}_{r}$ there exists a unique $1 \le m \le l$ such that $A \in \Gamma_{m}$ and so $A \subset K_{m}$. 
		We define $\tilde{K}_{m}=K_{m}+]-1,1[$ so that $d(K_{m}, \R \setminus \tilde{K}_{m} )=1$ and so assumption ${\bf (H9)}$ implies that 
		$$\mathbf F(K_{m}) \; {\rm and}\;  \mathbf F(\R \setminus \tilde{K}_{m} ) \textrm{ are independent}.$$
		Moreover, $\forall A \in \mathcal{K}_{r}$ , if $A \in \Gamma_{m}, $ then 
		$$ A \subseteq K_{m} \subseteq \tilde{K}_{m} \subseteq A +B_{1+\delta}. $$
		
		To define a martingale, we have to define a filtration. We define $\{\mathcal F_{\mu,t}\}_{t\ge 0}$ by $\mathcal F_{\mu,0}=\{\emptyset,\o\}$ and, for all $t>0$, 
		$$\{\mathcal F_{\mu,t}\}=\sigma\mbox{-algebra generated by }\oo \to H(p,y,\oo)\mathbf{1}_{\{\oo;y \in \mathcal R_{\mu,t}^{\oo} \}}, \quad y, p \in \R.$$ 
In particular, for all $t,s$ such that $0<t<s$ and $\oo \in \o $, we have:  $\mathcal R_{\mu,t}^{\oo} \subset \mathcal R_{\mu,s}^{\oo} $ so $\mathcal F_{\mu,t} \subset \mathcal F_{\mu,s} $ and then $\{\mathcal F_{\mu,t}$$\}_{t \ge 0}$ is indeed a filtration. 
		We also have for all $y \in \R$, that 
\begin{equation} \label{m mu est mesurable dans F mu t}
    \oo \to m_{\mu}(y,x,\oo)\mathbf{1}_{\{\oo; y \in \mathcal R_{\mu,t}^{\oo} \}}(\oo)\mbox{ is }\mathcal F_{\mu,t}\mbox{ measurable}.\end{equation}
  Indeed, this is a direct consequence of the fact that, using Lemma \ref{localistion de la solution du pmj}, we can restrict the representation of $m_{\mu}$ in $\mathcal R_{\mu,t}^{\oo}$.  

		\subsection{Proof of Theorem \ref{fluctuations}.}}\begin{proof}
		   We fix $x, y$ such that $|y- x| >1$  and we define $T=L_{\mu}|y-x|$ and $r=\frac{T}{l_{\mu}}$ where $l_{\mu}$ and $L_{\mu}$ are given in Proposition \ref{proprietes verifie par la soluton du probleme metrique}. Then,  using \eqref{l inclusion de sous niveau utilisant l equivalence a la distance euclidenne}, we have  
		\begin{equation}\label{inclusion du sous level de m mu dans la boule B r}
		    \mathcal R_{\mu,T}^{\oo}\subseteq [x-r,x+r].	\end{equation}
So, for every $\oo \in \o_{1}^\e$, we have that the $T$-sub level of $m_{\mu}(\cdot,x,\oo)$ is in $\mathcal {K}_{r}$. 
		
		The idea of the proof is based on Azuma's inequality, which gives the estimation of the variance. We define the martingale adapted to $\{\mathcal F_{\mu,t}\}$, for every $t \ge 0 $, by $$ X_{t}(\oo)=E[m_{\mu}(y,x,\oo)|\mathcal F_{\mu,t}]-E[m_{\mu}(y,x,\oo)].$$ Since, by \eqref{ineg.sol.metr.prob} $$  m_{\mu}(y,x,\oo) \le L_{\mu}|y-x| \le T , $$
we deduce, for all $t\ge T$, that $y \in \mathcal{R}_{\mu,T}^{\oo}\subseteq\mathcal{R}_{\mu,t}^{\oo}$ and so by \eqref{m mu est mesurable dans F mu t} $$m_{\mu}(y,x,\oo)=m_{\mu}(y,x,\oo)\mathbf{1}_{\{\oo; y \in \mathcal R_{\mu,t}^{\oo} \}}(\oo)\mbox{ is }\mathcal F_{\mu,t}\mbox{ measurable}.$$	This implies that	 \begin{equation}\label{la solution du pmj est mesurable pour tout les t plus grand que T}
 X_{0}=0 \mbox{ and } X_{t}(\oo)=m_{\mu}(y,x,\oo)-E\left[m_{\mu}(y,x,\oo)\right] ~\mbox{for all}~t \ge T.\end{equation}
To apply Azuma's inequality, we need to bound $|X_{t}-X_{s}|$. This is done in the following lemma

\begin{lem}\label{lemme qui nous donne la difference des martingales en fonction de t et s}We have the following inequality, for all $t,s>0$,
		$$\sup\limits_{\oo \in \o_1^\e } |X_{t}(\oo)-X_{s}(\oo)| \le L_\mu+2 \frac {L_\mu}{l_\mu}|s-t|.$$ 
\end{lem}
\begin{proof} In all the proof, $\oo$ is taken in $\o_1^\e$ (so that $\mu>\overline A=\tilde A(\oo)$). By \eqref{la solution du pmj est mesurable pour tout les t plus grand que T}, it suffices to show the result for $s,t \le T$. For all $0 < t \le s$, \eqref{m mu est mesurable dans F mu t} yields that 
		$$E\left[m_{\mu}(y,x,\cdot)\mathbf{1}_{\{\oo ; y \in \mathcal R_{\mu,t}^{\oo} \}}|\mathcal F_{\mu,s} \right]=m_{\mu}(y,x,\cdot)\mathbf{1}_{\{\oo;y \in \mathcal R_{\mu,t}^{\oo} \}}. $$  
This implies that
\begin{equation}\label{eq:01}
\begin{aligned}
X_{s}-X_{t}=&E\left[m_{\mu}(y,x,\cdot)|\mathcal F_{\mu,s} \right]  - E\left[m_{\mu}(y,x,\cdot)|\mathcal F_{\mu,t}\right]\\ 
=&E\left[m_{\mu}(y,x,\cdot)\mathbf{1}_{\{\oo, y \in \mathcal R_{\mu,t}^{\oo}\} }+m_{\mu}(y,x,\cdot)\mathbf{1}_{\{\oo, y \notin \mathcal R_{\mu,t} \}} |\mathcal F_{\mu,s}\right] \\
&-E\left[m_{\mu}(y,x,\cdot)\mathbf{1}_{\{\oo, y \in\mathcal R_{\mu,t}^{\oo}\} }+m_{\mu}(y,x,\cdot)\mathbf{1}_{\{\oo, y \notin \mathcal R_{\mu,t}\}}|\mathcal F_{\mu,t}\right]\\
=&E\left[m_{\mu}(y,x,\cdot)\mathbf{1}_{\{\oo, y \notin \mathcal R_{\mu,t} \}}| \mathcal F_{\mu,s}\right]-E\left[m_{\mu}(y,x,\cdot)\mathbf{1}_{\{\oo, y \notin \mathcal R_{\mu,t} \}}|\mathcal F_{\mu,t}\right].
		\end{aligned}
\end{equation}	
  Moreover, by \eqref{la solution du pmj comme somme de t plus min dur le sous niveau}, we have 
$$
 m_{\mu}(y,x,\cdot)\mathbf{1}_{\{\oo, y \notin \mathcal R_{\mu,t}\}}(\oo)=\left(t+m_{\mu}(y,\mathcal R_{\mu,t}^{\oo},\oo)\right)\mathbf{1}_{\{\oo, y \notin \mathcal R_{\mu,t}\}}(\oo)=t\cdot \mathbf{1}_{\{\oo, y \notin \mathcal R_{\mu,t}\}}(\oo)+m_{\mu}(y,\mathcal R_{\mu,t}^{\oo},\oo).
 $$
 
Injecting this in \eqref{eq:01}, and using that $t.\mathbf{1}_{\{\oo, y \notin \mathcal R_{\mu,t}^{\oo}\}}\in \mathcal{F}_{\mu,t} $, we get 
\begin{equation}\label{ difference ente X t et X s donne avec difference d esperance}
X_{s}-X_{t} = E\left[  m_{\mu}(y,\mathcal R_{\mu,t},\cdot )|\mathcal F_{\mu,s}  \right]-E\left[  m_{\mu}(y,\mathcal R_{\mu,t},\cdot )| \mathcal F_{\mu,t}  \right].
\end{equation}
Applying inequalities \eqref{inegalite.lipsh.compacts} and \eqref{distance entre les sous level}, we get
$$  m_{\mu}(y,\mathcal R_{\mu,t}^{\oo},\oo)-m_{\mu}(y,\mathcal R_{\mu,s}^{\oo},\oo) \le L_{\mu}dist_{\mathcal{H}}(\mathcal R_{\mu,t}^{\oo},\mathcal R_{\mu,s}^{\oo} ) \le \frac{L_{\mu}}{l_{\mu}}(s-t).  
$$
This implies that 
$$ X_{s}-X_{t} \le  \frac{L_{\mu}}{l_{\mu}}(s-t) +\left|      E\left[  m_{\mu}(y,\mathcal R_{\mu,s},\cdot)|\mathcal F_{\mu,s}  \right]-E\left[  m_{\mu}(y,\mathcal R_{\mu,t},\cdot )| \mathcal F_{\mu,t}  \right]  \right|.  
$$
 We now want to use the discretisation scheme to estimate  $E\left[  m_{\mu}(y,\mathcal R_{\mu,t},\cdot )|\mathcal F_{\mu,t}\right]$ by approximating the integral represented by the expectation as a sum of characteristic functions. 
   Inequality \eqref{inegalite.lipsh.compacts} yields that
   \begin{equation}\label{inegalite.R.et.ktilde}
0\le 		m_{\mu}(y,\mathcal R_{\mu,t},\cdot)-m_{\mu}(y,\tilde{K}_{m} ,\cdot) \le L_{\mu} dist_{\mathcal H}(\mathcal R_{\mu,t}^{\oo},\tilde{K}_{m}). 
\end{equation} 
 Using that
\begin{align*}
dist_{\mathcal H}(\mathcal R_{\mu,t}^{\oo},\tilde{K}_{m})\mathbf{1}_{\{ \oo; \mathcal R_{\mu,t}^{\oo} \in 
			\Gamma  _{m}\}}(\oo)\le& dist_{\mathcal H}(\mathcal R_{\mu,t}^{\oo},\mathcal R_{\mu,t}^{\oo}+B_{1+\d})\mathbf{1}_{\{ \oo; \mathcal R_{\mu,t}^{\oo} \in 
			\Gamma  _{m}\}}(\oo)\\ 
			\le& (1+\d)\mathbf{1}_{\{ \oo; \mathcal R_{\mu,t}^{\oo}  \in 
			\Gamma  _{m}\}}(\oo),
\end{align*}
we deduce that
\begin{align*}
 m_{\mu}(y,\tilde{K}_{m} ,\cdot)\mathbf{1}_{\{ \oo; \mathcal R_{\mu,t}  \in 
			\Gamma  _{m}\}} (\oo)&\le m_{\mu}(y,\mathcal R_{\mu,t},\cdot)\mathbf{1}_{\{ \oo;\mathcal R_{\mu,t}\in 
			\Gamma  _{m}\}} (\oo)\\
			& \le \left(L_{\mu}(1+\d)  + m_{\mu}(y,\tilde{K}_{m} ,\cdot)\right) \mathbf{1}_{\{ \oo; \mathcal R_{\mu,t}\in 
			\Gamma  _{m}\}}(\oo).
\end{align*} 
We now want to apply the conditional expectation. We recall that the result of {{\cite[Lemma 4.2]{homog.stoch.quantitative.premier.ordre.concexe}}}, imply that for each $1\le m \le l, t>0$, and $A\in \mathcal F_{\mu,t}$, we have that 
$$\mathbf{1}_{A\cap\{ \oo; \mathcal{R}_{\mu,t}^{\oo}\in \Gamma_{m} \}}\quad{\rm is}\; \textbf{F}(K_{m})\textrm{-measurable}$$
 and 
 $$E \left[ m_{\mu}(y,\tilde{K}_{m} ,\cdot)\mathbf{1}_{\{ \oo; \mathcal R_{\mu,t} \in \Gamma  _{m}\}}  |\mathcal F_{\mu,t}   \right] =E \left[ m_{\mu}(y,\tilde{K}_{m} ,\cdot)\right]\mathbf{1}_{\{ \oo; \mathcal R_{\mu,t} \in \Gamma_{m}\}}.$$
 We then have
 \begin{equation}\label{inegalite.esp.cond.ktilde.Ra}
		\begin{aligned}
&E \left[ m_{\mu}(y,\tilde{K}_{m} ,\cdot)\right]\mathbf{1}_{\{ \oo; \mathcal R_{\mu,t} \in 
			\Gamma_{m}\}} \\ 
&=E \left[ m_{\mu}(y,\tilde{K}_{m} ,\cdot)\mathbf{1}_{\{ \oo; \mathcal R_{\mu,t} \in 
			\Gamma  _{m}\}}    |\mathcal F_{\mu,t}   \right]\\ 
		& \le  E \left[ m_{\mu}(y,\mathcal R_{\mu,t},\cdot)\mathbf{1}_{\{ \oo;\mathcal  R_{\mu,t} \in 
			\Gamma_{m}\}}|\mathcal F_{\mu,t}\right] \\ & \le \left(L_{\mu}(1+\d)+E\left[  m_{\mu}(y,\tilde{K}_{m} ,\cdot)\right]\right)  \mathbf{1}_{\{ \oo; \mathcal R_{\mu,t} \in \Gamma_{m}\}} .
		\end{aligned}\end{equation} 
		The set $\{\Gamma_{m}\}$ is a partition of $\mathcal{K}_{r}$ and $\mathcal R_{\mu,T} ^{\oo}\in \mathcal{K}_{r}$ so for every $1 \le t \le s \le T$ we have $\mathbf{1}_{\{ \oo;\mathcal R_{\mu,t}^{\oo} \in 
			\mathcal{K}_{r}\}}=\mathbf{1}_{\{ \oo; \mathcal R_{\mu,s}^{\oo} \in 
			\mathcal{K}_{r}\}}=1$ and so 
$$ \begin{aligned}
		 E\left[m_{\mu}(y,\mathcal R_{\mu,t},\cdot)|\mathcal F_{\mu,t}\right]&=E\left[m_{\mu}(y,\mathcal R_{\mu,t},\cdot)\mathbf{1}_{\{ \oo;\mathcal R_{\mu,t} \in 
			\mathcal {K}_{r}\}} |\mathcal F_{\mu,t}\right]\mathbf{1}_{\{ \oo; \mathcal R_{\mu,s} \in 
			\mathcal{K}_{r}\}} \\&=
		\sum_{m,j=1}^{l}E\left[m_{\mu}(y,\mathcal R_{\mu,t},\cdot)\mathbf{1}_{\{ \oo;\mathcal  R_{\mu,t} \in 
			\Gamma  _{m}\}} |\mathcal F_{\mu,t}\right]\mathbf{1}_{\{ \oo; \mathcal R_{\mu,s} \in 
			\Gamma  _{j}\}}.
		\end{aligned} $$Multiplying inequality \eqref{inegalite.esp.cond.ktilde.Ra} by $\mathbf{1}_{\{ \oo; \mathcal R_{\mu,s} \in 
			\Gamma  _{j}\}} $ and applying the summation on $m,j$ we obtain
		$$\begin{aligned}
		&  \sum_{m,j=1}^{l}E\left[m_{\mu}(y,\tilde{K}_{m},\cdot)\right]\mathbf{1}_{\{ \oo; \mathcal R_{\mu,t} \in 
			\Gamma  _{m}\}} \mathbf{1}_{\{ \oo;\mathcal R_{\mu,s}\in 
			\Gamma  _{j}\}}\\&  \le\sum_{m,j=1}^{l}E\left[m_{\mu}(y,\mathcal R_{\mu,t},\cdot)\mathbf{1}_{\{ \oo;\mathcal R_{\mu,t} \in 
			\Gamma  _{m}\}} |\mathcal F_{\mu,t}\right]\mathbf{1}_{\{ \oo; \mathcal R_{\mu,s} \in 
			\Gamma  _{j}\}}=E\left[m_{\mu}(y,\mathcal R_{\mu,t},\cdot)|\mathcal F_{\mu,t}\right] \\ & \le L_{\mu}(1+\d)+\sum_{m,j=1}^{l}E\left[  m_{\mu}(y,\tilde{K}_{m} ,\cdot)\right]  \mathbf{1}_{\{ \oo; \mathcal R_{ \mu,t}\in \Gamma_{m}\}}\mathbf{1}_{\{ \oo; \mathcal R_{ \mu,s} \in 
			\Gamma  _{j}\}}. \end{aligned}$$ Finally we conclude that $$ 0 \le E\left[m_{\mu}(y,\mathcal R_{\mu,t},\cdot)|\mathcal F_{\mu,t}\right]-\sum_{m,j=1}^{l}E\left[  m_{\mu}(y,\tilde{K}_{m} ,\oo)\right]  \mathbf{1}_{\{ \oo; \mathcal R_{\mu,t}\in \Gamma_{m}\}}\mathbf{1}_{\{ \oo; \mathcal R_{\mu,s} \in 
			\Gamma  _{j}\}} \le  L_{\mu}(1+\d).$$ We repeat the same proof by interchanging the role of $t$ and $s$ to obtain$$ 0 \le E\left(m_{\mu}(y,\mathcal R_{\mu,s},\cdot)|\mathcal F_{\mu,s}\right)-\sum_{m,j=1}^{l}E\left(  m_{\mu}(y,\tilde{K}_{j} ,\oo)\right)  \mathbf{1}_{\{ \oo; \mathcal R_{\mu,t}\in \Gamma_{m}\}}\mathbf{1}_{\{ \oo; \mathcal R_{\mu,s}\in 
			\Gamma  _{j}\}} \le  L_{\mu}(1+\d).$$ The last two inequalities yield that \begin{equation}\label{inegalite sur les esperences conditionnelles}
			    \begin{aligned}&
		\left|E\left[m_{\mu}(y,\mathcal R_{\mu,s},\cdot)|\mathcal F_{\mu,s}\right]-E\left[m_{\mu}(y,\mathcal R_{\mu,t},\cdot)| \mathcal F_{\mu,t}\right]\right|\\& \le  L_{\mu}(1+\d)+\sum_{m,j=1}^{l}\left|E\left[  m_{\mu}(y,\tilde{K}_{m} ,\cdot)\right]- E\left[  m_{\mu}(y,\tilde{K}_{j} ,\cdot)\right]\right| \mathbf{1}_{\{ \oo;\mathcal R_{\mu,t}\in \Gamma_{m}\}}\mathbf{1}_{\{ \oo; \mathcal R_{\mu,s} \in 
			\Gamma  _{j}\}}   \end{aligned}	\end{equation}If for some $m,j =1,\dots, l$, there exists $\oo$ such that $\mathcal R_{\mu,t}^{\oo} \in \Gamma_{m}$ and $\mathcal R_{\mu,s}^{\oo} \in \Gamma_{j}$ then 
			$$\mathcal R_{\mu,t}^{\oo} \subseteq K_{m}\subseteq \tilde{K}_{m} \subseteq  \mathcal R_{\mu,t}^{\oo}+B_{1+\d}  $$ and $$\mathcal R_{\mu,s}^{\oo} \subset K_{j} \subseteq \tilde{K}_{j} \subseteq  \mathcal R_{\mu,s}^{\oo}+B_{1+\d},$$ 
which implies
$$ dist_{\mathcal H}(\tilde{K}_{m},\tilde{K}_{j}) \le dist_{\mathcal H} (K_{m},K_{j})\le dist_{\mathcal H}(\mathcal R_{\mu,t}^{\oo},\mathcal R_{\mu,s}^{\oo})+2\d \le \frac{1}{l_{\mu}} (s-t)+2\d.$$
		We then get that for all $m,j=1,\dots,l$ $$\begin{aligned}&\left|E\left[  m_{\mu}(y,\tilde{K}_{m} ,\cdot)\right]- E\left[  m_{\mu}(y,\tilde{K}_{j} ,\cdot)\right]\right| \mathbf{1}_{\{ \oo; \mathcal R_{\mu,t}\in \Gamma_{m}\}}\mathbf{1}_{\{ \oo; \mathcal R_{\mu,s} \in 
			\Gamma  _{j}\}} \\&\le L_{\mu}dist_{\mathcal H}(\tilde{K}_{m},\tilde{K}_{j})\mathbf{1}_{\{ \oo; \mathcal R_{\mu,t}\in \Gamma_{m}\}}\mathbf{1}_{\{ \oo; \mathcal R_{\mu,s} \in 
			\Gamma  _{j}\}} \\&\le \left(\frac{L_{\mu}}{l_{\mu}}(s-t)+2L_{\mu}\d \right)\mathbf{1}_{\{ \oo; \mathcal R_{\mu,t}\in \Gamma_{m}\}}\mathbf{1}_{\{ \oo;\mathcal  R_{\mu,s} \in 
			\Gamma  _{j}\}}. \end{aligned} $$
	Injecting this in \eqref{inegalite sur les esperences conditionnelles}, we conclude that, for every $\d>0$  
$$\begin{aligned}&
		\bigg|E\left[m_{\mu}(y,\mathcal R_{\mu,s},\cdot)|\mathcal F_{\mu,s}\right]-E\left[m_{\mu}(y,\mathcal R_{\mu,t},\cdot)|\mathcal F_{\mu,t}\right]\bigg| \\& \le L_{\mu}(1+\d) +\left(\frac{L_{\mu}}{l_{\mu}}(s-t)+2L_{\mu}\d\right).\end{aligned} 
$$
Using \eqref{ difference ente X t et X s donne avec difference d esperance} and sending $\d \to 0$, we finally get
\begin{equation}\label{derniere inegalite sur l esperence cond qui nous donne azzuma}
		|X_{s}-X_{t}| \le L_{\mu}+2\frac{L_{\mu}}{l_{\mu}}(s-t). 
\end{equation}
\end{proof}
		The idea is now to apply Azuma's inequality. For this reason we define the following martingale 
$$ \tilde{X}_{k}=X_{hk} ~\mbox{with}~h=\frac{l_{\mu}}{2}.$$  
		Lemma \ref{lemme qui nous donne la difference des martingales en fonction de t et s} implies that $$ |\tilde{X}_{k+1}-\tilde{X}_{k}| \le  2L_{\mu} .$$  Applying Azuma's inequality, Proposition \ref{proposition de azuma}, we obtain, for all $\lambda>0$ and $N \in \mathbb{N}$, 
$$P(\{\oo\in \o_1^\e,\; |\tilde{X}_{N} (\oo)|> \lambda\}) \le exp\left(\frac{-\lambda ^{2}}{8 L_{\mu}^{2}N}\right).  $$Let $N \in \mathbb{N}$ such that $\frac{T}{h}\le  N<\frac{T}{h}+1$. We have$$ \tilde{X}_{N}=X_{T}=m_{\mu}(y,x,\cdot)-E[m_{\mu}(y,x,\cdot)].$$ 
		It follows that (recall that $T = L_{\mu}|y-x|$ and $|y-x|>1$)
		$$   N \le \frac{T}{h}+1 \le \frac{l_{\mu}+2L_{\mu}}{l_{\mu}}|y-x|.$$
We then deduce that 
\begin{align*}
P\left(\{\oo\in\o_1^\e,\; |m_{\mu}(y,x,\oo)-E[m_{\mu}(y,x,\oo)]|>\lambda \} \right]& \le exp\left(  \frac{-\lambda^{2}l_{\mu}}{8L_\mu^2(2L_\mu+l_\mu)|y-x|} \right)\\
&\le exp \left(  \frac{-\lambda^{2}c'(\mu-\overline{A})}{8L_\mu^2(2L_\mu+l_\mu)|y-x|} \right) .
\end{align*}
This completes the proof of Theorem \ref{fluctuations}. \end{proof}

\subsection{Proof of Theorems \ref{homo.du.pmj} and \ref{homo.du.pmj-bis} }	
\begin{proof}[Proof of Theorem \ref{homo.du.pmj}]	
	For all $\mu >\overline{A}$, $\overline{\lambda}=t\lambda>0$ and $ \overline{y}=t y,~\overline{x}=tx $ such that $|y-x|\ge \frac 1 t$,  Theorem \eqref{fluctuations} yields that 
	$$ P\left(\left\{\oo\in \o_1^\e,\; \frac{1}{t}\left|m_{\mu}(ty,tx,\oo)-E[m_{\mu}(ty,tx,\oo)]\right|>\lambda\right\}\right)\le exp\left(-\frac{t\lambda^{2}(\mu-\overline A)}{C|y-x|}\right). $$ 
Then,  using the following proposition
		\begin{prop}{{\cite[Proposition 8.4]{conver en proba implique presque sur} }}
			Let $Y$ and $(Y_{n})_{n\ge1}$ real random variables defined on $(\o,\textbf F,P)$ and verify for all $\eps>0$,
			$$  \sum_{n\ge1}P(|Y_{n}-Y|\ge \eps)<+\infty .$$Then $Y_{n}\to Y$ a.s.
		\end{prop}
We deduce, for all $x, y\in\R$, that there exists $\o_0^\e\subset \o_1^\e$ with $P(\o_0^\e)= P(\o_1^\e)=1-\theta(\e)$ such that
 \begin{equation}\label{convergence presque partout de la solution du probleme metrique}
	\frac{1}{t} \left|m_{\mu}(ty,tx,\oo)-E[m_{\mu}(ty,tx,\oo)]\right|\to 0 \mbox{ for all }\oo\in \o_0^\e.\end{equation}

	This means that the random process $\frac{1}{t}m_{\mu}(ty,tx,\oo)$ is very close to his average for large $t$. So to conclude the proof, it suffices to prove that this average converges to a deterministic quantity which we denoted by $\overline{m}_{\mu}$. We know that the solution of the metric problem is sub-additive and then his average too. Moreover by \eqref{ineg.sol.metr.prob}, $m_{\mu}(y,x,\cdot)\ge 0$  and so $E(m_{\mu}(y,x,\cdot)) \in \R^{+}\cup \{ +\infty\}$.
		
		We now use the following lemma.
		\begin{lem}[Fekete's Lemma]{{\cite[Lemma 1.4]{Fekete}}}
			Let ${\displaystyle (u_{n})_{n\geq 1}}$ be a sub-additive sequence. Then the limit, when $n$ goes to $+\infty$, of the sequence ${\displaystyle (u_{n}/n)_{n\geq 1}}$ exists and we have
			$${\displaystyle \lim _{n\rightarrow +\infty }{\dfrac {u_{n}}{n}}=\inf _{n\geq 1}{\dfrac {u_{n}}{n}}\in \mathbb {R} \cup \{-\infty \}.} $$
		\end{lem} 
		This implies  that the limit of $\frac{E(m_{\mu}(ty,tx,\oo))}{t}$ exists and we denoted it by $\overline{m}_{\mu}(y,x)$. Injecting this in \eqref{convergence presque partout de la solution du probleme metrique}, we get, for all $\mu>\overline{A}$ and $x,y\in \R$, that
 $$  \frac{m_{\mu}(ty,tx,\oo)}{t} \to \overline{m}_{\mu}(y,x)\mbox{ almost surely}.$$
 This ends the proof of Theorem\ \ref{homo.du.pmj}
\end{proof}
In order to prove Theorem \ref{homo.du.pmj-bis}, we need to determine the equation satisfied by $\overline m$, when $x\ge 0,y>0$ (resp. $x\le 0,y<0$).

		\begin{prop}\label{proposition de la limite de la solution du probleme limite}
		For all $\mu >\overline{A}$, and $x \ge 0$ we have that $\overline{m}_{\mu}(\cdot,x) \in \mathbf{Lip} $ is a solution of 
		\begin{equation}\label{equation limite du probleme metrique}
		 \left\{
\begin{aligned}
    &\overline{H}_{R}(D\overline{m}_{\mu}(y,x))=\mu &\mbox{ for } y\neq x,&~ y>0  ,\\
    & \overline{m}_{\mu}(x,x)=0.
\end{aligned}		 
		 \right.
		\end{equation}
		\end{prop}
		\begin{proof} We prove that $\overline{m}_{\mu}(\cdot,x)$ is a sub-solution for \eqref{equation limite du probleme metrique}. The proof of super-solution is similar so we skip it. Let $\phi$ be a test function such that $\overline{m}_{\mu}(\cdot,x)-\phi(\cdot)$ attains a strict maximum point at $\overline{y}\neq x$, i.e $$  (\overline{m}_{\mu}(\cdot,x)-\phi(\cdot))(y)<(\overline{m}_{\mu}(\cdot,x)-\phi(\cdot))(\overline{y})=0\mbox{ for }y \in B_{\overline{r}}(\overline{y})\setminus\{ \overline{y}\}$$
for $\overline r$ small enough such that $B_{\overline r}(\overline y)\subset ]0,+\infty[\setminus\{x\}$.
Since the maximum is strict, we assume that \begin{equation}\label{existence de kr}
		    (\overline{m}_{\mu}(\cdot,x)-\phi(\cdot))(y)\le -k_{r} \mbox{ for all }y \in \partial B_{\overline{r}}(\overline{y})
		\end{equation}
for some $k_r>0$ small enough. 
		We argue by contradiction, by assuming that \begin{equation}\label{contradiction que la solution limite du pmj ne verifie pas le pmj limite}
	\overline{H}_{R} (\phi'(\overline{y}))-\mu=\theta>0. 	\end{equation} 
		Let $p=\phi'(\overline{y})$ and $m_{\mu+\theta}^{R}(\cdot,x,\oo)$ be a solution of the metric problem, far away from the junction point \eqref{P.M.LOIN.J}, with $H_{\a}=H_{R}$ and $\{0\}$ replaced by $\{x\}$.
		We also know that $v^{R}(y,x,\oo)=m_{\mu+\theta}^{R}(y,x,\oo)-p(y-x)$ is a solution of \eqref{P.M.LOIN.J.avec.p}. In particular it's a solution on $]0,+\infty[\setminus\{x\}$. We claim that, if $\eta>0$ is small enough, then the perturbed test function 
		$$    \phi_{\eta}(y,x)=\phi(y)+\eta v^{R}\left(\frac{y}{\eta},\frac{x}{\eta},\oo\right)     $$
is a super-solution of 
$$   H\left(\phi_{\eta}'(y,x),\frac{y}{\eta},\oo\right)-\mu\ge \frac{\theta}{2}\mbox{ in }B_{\overline r}(\overline{y}),$$
for $\overline{r}>0$ small enough. To see this, let $\psi$ be a test function touching $\phi_{\eta}$ from below at $y_{1}\in B_{\overline r}(\overline{y})$. We have that $\psi(y_{1})=\phi_{\eta}(y_{1})$. So we obtain that the function 
$$ \zeta(y)=\frac{1}{\eta}(\psi(\eta y)-\phi(\eta y))  
$$ 
touches $v^{R}$ from below at $y_{1}$. Then we obtain that 
$$ 
H_{R}\left(\psi '(y_{1})-\phi '(y_{1})+p,\frac{y_{1}}{\eta},\oo\right) \ge \overline{H}_{R}(p).
$$ 
We deduce, using \eqref{contradiction que la solution limite du pmj ne verifie pas le pmj limite}, the continuity of $\phi'$ and the fact that $H_{R}$ is Lipschitz continuous, that for $\eta$ and $\overline r$ small enough (in particular such that $\frac{y_1}{\eta}>\frac{2}{\sqrt \e}$)
		$$ \begin{aligned}
		    &	H\left(\psi' (y_{1}),\frac{y_{1}}{\eta},\oo\right)-\mu=	H_{R}\left(\psi '(y_{1}),\frac{y_{1}}{\eta},\oo\right)-\mu\\&
		 \ge 
		H_{R}\left(\psi '(y_{1}),\frac{y_{1}}{\eta},\oo\right)+\theta-H_{R}\left(\phi'(\overline{y})+\psi '(y_{1})-\phi '(y_{1}),\frac{y_{1}}{\eta},\oo\right)\\
		& \ge \theta/2.
			\end{aligned} 
$$ 
So the claim is proved. Using \eqref{existence de kr}, we deduce, for $\eta$ small enough, that 
$$ 
\eta m_{\mu}\left(\frac{y}{\eta},	\frac{x}{\eta},\oo\right)+\frac{k_{r}}{2}\le \phi_{\eta}\mbox{ for }y \in\partial B_{\overline r}(\overline{y}).
$$
So by comparison principle we deduce that 
$$ 
\eta m_{\mu}\left(\frac{y}{\eta},	\frac{x}{\eta},\oo\right)+\frac{k_{r}}2\le \phi_{\eta}\mbox{ for }y \in B_{r}(\overline{y}). 
$$ 
Now passing to limit $\eta\to 0$ and taking $y = \overline{y}$ we obtain that
 $$  \overline{m}_{\mu}(\overline{y},x)<\overline{m}_{\mu}(\overline{y},x)+\frac{k_{r}}2\le \phi(\overline{y})=\overline{m}_{\mu}(\overline{y},x).$$ This is a contradiction with the definition of $k_{r}>0$. This ends the proof.\end{proof}
		
In the same way, we have the following proposition for $x\le 0,y<0$
		\begin{prop}\label{proposition de la limite de la solution du probleme metrique pour x <0}
		For $x \le 0$, and for all $\mu >\overline{A}$, we have that $\overline{m}_{\mu}(\cdot,x) \in \mathbf{Lip} $ is a solution of 
		\begin{equation}
		 \left\{
\begin{aligned}
    &\overline{H}_{L}(D\overline{m}_{\mu}(y,x))=\mu &\mbox{ for } y\neq x,&~y< 0\\
    & \overline{m}_{\mu}(x,x)=0& \mbox{ for }y=x.
\end{aligned}		 
		 \right.
		\end{equation}
	\end{prop}	
	\bigskip

We are now able to give the proof of Theorem \ref{homo.du.pmj-bis}
\begin{proof}[Proof of Theorem \ref{homo.du.pmj-bis}]
We only prove the case $x,y\ge0$, the other one being similar. To prove the result, it is enough to prove that
\begin{equation}\label{representation de la limite de pmj}
 \overline{m}_{\mu}(y,x)=\sup\left\{ p(y-x);~ 
\overline{H}_{R}(p)\le \mu 
 \right\} . \end{equation}

Inequality \eqref{ineg.sol.metr.prob} yields that 
$$ 0
<l_{\mu} \le \frac{\overline{m}_{\mu}(y,x)}{|y-x|}\le L_{\mu}
,$$ 
then $\liminf\limits_{y \to +\infty}\frac{\overline{m}_{\mu}(y,x)}{|y-x|}>0$ and so the comparison principle (Proposition \ref{principe de comparaison du probleme metrique sur la jonction} reformulated with $\{0\}$ replaced by $\{x\}$) holds. Since the right hand side of \eqref{representation de la limite de pmj} is a sub-solution of \eqref{equation limite du probleme metrique} then Proposition \ref{principe de comparaison du probleme metrique sur la jonction} yields that
 $$ \overline{m}_{\mu}(y,x)\ge \sup\left\{ p(y-x); ~
\overline{H}_{R}(p)\le \mu  \right\}.   $$
     If the reverse inequality doesn't hold, and since $\overline{m}_{\mu}(\cdot,x)$ is a Lipschitz continuous function, we can find $y >0$; $y \neq x $ such that $\overline{m}_{\mu}(\cdot,x)$ is differentiable at $y $ and   $$ \overline{m}_{\mu}(y,x)> \sup\left\{ p(y-x); ~\overline{H}_{R}(p)\le \mu   \right \}.   $$ 	We also have, immediately from the form of the limit of $\frac{m_{\mu}(ty,tx)}{t}$, that $\overline{m}_{\mu}$ must be positively homogeneous. The sub-additive property of $m_{\mu}$ easily translates into subadditivity property for $\overline{m}_{\mu}$ and therefore $\overline{m}_{\mu}$ is convex. The function $\overline{m}_{\mu}(\cdot,x)$ is differentiable at $y $ then $D\overline{m}_{\mu}(y,x)$ is in the sub- and the super differential of a convex function, then $$ \overline{m}_{\mu}(y,x)=D\overline{m}_{\mu}(y,x)(y-x). $$ This follows that $$ 
\overline{H}_{R}(D\overline{m}_{\mu}(y,x))> \mu .$$ This contradicts the fact that $\overline{m}_{\mu}(\cdot,x)$ is a solution of \eqref{equation limite du probleme metrique}. This ends the proof.
    \end{proof}

    Combining Theorem \ref{homo.du.pmj}, Theorem \ref{homo.du.pmj-bis}, Proposition \ref{proposition de la limite de la solution du probleme limite} and Proposition \ref{proposition de la limite de la solution du probleme metrique pour x <0}  for $x \le 0$ and $x\ge 0$, we get the following result for $x=0$.
    \begin{prop}
    For all $\mu>\overline{A}$, we have that $$ \frac{m_{\mu}(ty,0,\oo)}{t}\to \overline{m}_{\mu}(y,0)=\overline{m}_{\mu}(y)\mbox{ for all  }\oo\in \o_0^\e.  $$The function $\overline{m}_{\mu}(\cdot)$ is a solution of \begin{equation}\label{probleme metrique limite en 0}
    \left\{
    \begin{aligned}
    &\overline{H}_{R}(D\overline{m}_{\mu}(y,0))=\mu &y>0,\\
     &\overline{H}_{L}(D\overline{m}_{\mu}(y,0))=\mu &y<0,\\
   &  \overline{m}_{\mu}(0,0)=0,
    \end{aligned}
    \right.
       \end{equation}
 and  is given by 
 $$\overline{m}_{\mu}(y)=\left\{
 \begin{aligned}
 &\overline{p}^{R,+}_{\mu}\cdot y &y>0;\\
 &\overline{p}^{L,-}_{\mu}\cdot y &y<0.
 \end{aligned}
 \right.
 $$\end{prop}

\subsection{Some properties of $m_{\mu}(x,y)$ }
In this subsection we give some properties of the function $m_\mu(x,y)$. To do this, we consider the hamiltonian $G:\R\times \R\times \o\to \R$ given by 
$$ G(p,y,\oo)=H(-p,y,\oo) . $$
The function $G$ verifies assumptions ${\bf (H)}$, and the definition of $\overline{A}$ is the same. 

Proceeding in the same way, we can prove that there exists, for $x\in \R$, and for all $\mu \ge \overline{A}$ a solution $n_{\mu}(\cdot,x,\oo)$ solution of \begin{equation}\label{HJ G}
\left\{    \begin{aligned}
&G(Dn_{\mu}(y,x,\oo),y,\oo)=\mu,& \mbox{ for }y \neq x;\\
&n_{\mu}(x,x)=0.
    \end{aligned}\right.
\end{equation} 
Moreover, $n_{\mu}$ is given by 
$$n_{\mu}(y,x)=\sup\{  v(y,\oo)-v( x,\oo    ); ~v\in \textbf{Lip};~ G(Dv,y,\oo)\le \mu \mbox{ in }\R    \},    $$
and so
 $$ n_{\mu}(y,x)=m_{\mu}(x,y). $$
We also have that, for $x \ge 0,~y>0$ (resp. $x \le 0,~y<0$) that 
$$\frac{n_{\mu}(ty,tx)}{t}\to \overline{n}^{R}_{\mu}(y-x)=\overline{m}_{\mu}^{R}(x-y),      $$ 
$$\left(\textrm{resp. } \frac{n_{\mu}(ty,tx)}{t}\to \overline{n}^{L}_{\mu}(y-x)=\overline{m}_{\mu}^{L}(x-y)\right).$$
In particular, 
 $$ \overline{n}^{R}_{\mu}(y-x)=\overline{m}^{R}_{\mu}(x-y)=
\left\{
\begin{aligned}
&-\overline{p}_{\mu}^{R,-}(y-x) & \mbox{ for }y>x;\\
&-\overline{p}_{\mu}^{R,+}(y-x) & \mbox{ for }y<x.
\end{aligned}
\right.
$$
 \section{Main results for the approximated corrector at the junction point}\label{les resultats du correcteur app}
 For stochastic homogenization, it is well known that correctors don't exist in general and one have to introduce approximated correctors. 
 At the junction, using the particular form of the test function (see  \cite[Theorem 2.7]{Imbert}) and  the ansatz
 $$
  u^{\eps}(t,x)=u(t,0)+\eps v^{\d}(\eps^{-1}y,\oo)
 $$
 the approximated corrector $v^\delta$ has to satisfy
 \begin{equation}\label{correcteur app}
		\d v^{\d}(y,\oo)+H(Dv^{\d},y,\oo)=0 ~~y \in \R,
\end{equation}
$$\d v^{\d}(y,\oo)   \textrm{ converges to }-\overline{A} \textrm{ in  balls of radius }1/\d 
$$
and 
$$\d v^{\d}\left(\frac{y}{\d},\oo\right)\to W(y)$$
where $W(0)=0$  and 
\begin{equation}\label{eq:W}
\overline{p}_{R}^{+}y\mathbf{1}_{\{y >0\}} +\overline{p}^{-}_{L}y \mathbf{1}_{\{y<0\}} \le W(y)\le\hat{p}_{R}^{+}y\mathbf{1}_{\{y >0\}} +\hat{p}^{-}_{L}y \mathbf{1}_{\{y<0\}}
\end{equation}
with 
\begin{equation}\label{eq:p+p-}
\left\{\begin{array}{ll}
\overline p_L^-=\sup\{p, \; \overline H_L(p)=\overline A\}\\
\hat p_L^-=\inf\{p, \; \overline H_L(p)=\overline A\}
\end{array}
\right.
\quad {\rm and}
\quad 
\left\{\begin{array}{ll}
\overline p_R^+=\inf\{p, \; \overline H_R(p)=\overline A\}\\
\hat p_R^+=\sup\{p, \; \overline H_R(p)=\overline A\}.
\end{array}
\right.
\end{equation}
 Note in particular that if $\overline A>\mu^\star_R$, then, by convexity of $\overline H_R$,  $\overline p_R^+=\hat p_R^+=\overline p_{\overline A}^{R,+}$ and the same result holds for $\overline p^-_L$ and $\hat p^-_L$.
 Theses results are made precise in  the following theorem.
 \begin{theorem}[Approximated correctors]\label{convergence de correcteur approche dans une boule et quil verifie les bonnes pentes}
			Assume ${\bf (H}$-${\bf \e)}$ (resp. ${\bf (H}$-${\bf WFL)})$ and let $\e_0>0$ small enough. Then for all $0<\d \le \e\le \e_0$, $r>0$, there exists a solution $v^{\d}$ of \eqref{correcteur app} such that, for all $\oo\in \o_{0}^{\e_0}$ (resp $\oo\in \o_0$), $\d v^{\d }(y,\oo)$ converges locally uniformly in $y\in B_{r/\d}$  to $-\overline{A}$, i.e, we have that
$$
\lim\limits_{ \d \to 0}|\d v^{\d}(y,\oo)+\overline{A}|=0 \mbox{ and for all }y \in B_{r/\d}.
$$ 
In addition, if we define $\overline v^\d$ by  $\overline v^{\d}(y,\oo)=\d v^{\d}(\d^{-1}y,\oo)$, then $\overline v^\d$ converges locally uniformly, as $\d \to 0$, to a function $W$ which satisfies \eqref{eq:W}.
\end{theorem}
	In the rest of this section, we will only treat the case of assumption ${\bf (H}$-${\bf \e)}$, the case of assumption ${\bf (H}$-${\bf WFL)}$ being similar and even simpler.

\subsection{Existence of approximated correctors at the junction point}\label{sous section des resultats du correcteur approche}This sub-section is devoted to the existence of approximated correctors at the junction point. We will also prove that $\d v^\d$ converges to $-\overline A$ in balls of radius $1/\d$.  We begin by the existence result.
\begin{proposition}[Existence of approximates correctors]\label{pro:existence-vd}
Assume  ${\bf (H}$-${\bf \e)}$ and let $\e_0>0$. Then, for every $\oo \in \o_{0}^{\e_0} $ and $0<\d \le \e\le \e_0$, the problem \eqref{correcteur app} has a unique bounded solution $v^{\d}$ which satisfies
$$| \d v^{\d}|\le E \mbox{ in }\R,$$
where $E\ge\sup\limits_{y\in \R}|H(0,y,\oo)|$ ($E $ is finite and independent of $\oo$ by assumption \textbf{(H2)}). 
Moreover, $v^{\d} $ is Lipschitz continuous in $\R$ uniformly in $\d$.
\end{proposition}

\begin{proof}
 The solution is constructed by Perron's method. Indeed,  $\pm E /\d $ are super- and sub-solution of \eqref{correcteur app}. Then by Perron's method, there exists a solution $v^{\d}$ of \eqref{correcteur app} such that 
 $$| \d v^{\d}|\le E\mbox{ in }\R.$$
 In addition, by coercivity of $H$, we have that $v^{\d} $ is Lipschitz continuous in $\R$ uniformly in $\d$. 
 
\end{proof}

To prove the convergence of $\d v^{\d}$ in balls of radius $1/\d$, we first demonstrate that this convergence holds at zero. This result is given in  the next proposition and the idea  is to compare $ v^{\d} $ and $ m_{\overline {A}+\a} $, for $\a >0$ small enough.

\begin{proposition}\label{convergence du correcteur approvhe a partir du probleme metrique } 
For all $\oo \in \o_{0}^{\e_0}$, we have that 
$$ \d v^{\d}(0,\oo) \to -\overline{A}. $$
\end{proposition}
In order to prove Proposition \ref{convergence du correcteur approvhe a partir du probleme metrique }, we need the following result.
		\begin{proposition}[{{\cite[Proposition 2.15]{Imbert}}}]\label{restriction des solutions} Let $H: \R \times \R \to \R$, $(p,y) \to H(p,y)$, be a Lipschitz continuous, quasi-convex and coercive function with respect to $p$ and let $a,b\in \R$. Let $u:(0,T)\times \R\to \mathbb{R}$.  If $u$ satisfies 
			$$u_{t}(t,y)+H(Du,y)=0 ~\mbox{for}~(t,x)\in (0,T)\times \R,$$ then $u$ satisfies
			$$\left\{\begin{aligned} &u_{t}+H(Du,y)= 0 ~&\mbox{for}~(t,x)&\in (0,T)\times (a,b),\\
			&u_{t}+H^{-}(Du,y)= 0~&\mbox{for}~(t,x)&\in (0,T)\times \{a\},
			\\
			&u_{t}+H^{+}(Du,y)= 0~&\mbox{for}~(t,x)&\in (0,T)\times \{b\},
			\end{aligned}\right.
			$$
			and $$u(t,c)=\limsup\limits_{(s,y)\to(t,c),y \in (a,b)}u(s,y)\mbox{ for }c=a,b.$$
		\end{proposition}
\begin{proof}[Proof of Proposition \ref{convergence du correcteur approvhe a partir du probleme metrique }]	
To prove the convergence of $\d v^{\d}(0,\oo)$ to $-\overline{A}$, we prove, for $\a$ small enough, that
$$|\overline{A}+ \d v^{\d}(0,\oo)|\le C \a .
$$  
First, by Proposition \ref{pro:existence-vd}, we have 
$$ -\overline{A}\le \d v^{\d}(y,\oo). 
$$
Now we want to prove the reverse inequality. The idea of the proof is to compare  $m_{\mu}(x,\cdot,\oo)$, for $\mu$ closed to $\overline{A}$, and $v^{\d}$. 
We set  $\mu =\overline{A}+\a$, with $\a \le \frac{1}{16}$ and  $y_{0} =2r$, with $r>0$ to be defined later. We then define the following function 
 $$ \psi(x)= v^{\d}(x,\oo)+m_{\mu}(-y_{0},x,\oo)+\a x.
 $$
 This function have a maximum in $ [-r,0]$ denoted by $M$ and reached in $\overline{x}$.
 We now want to show that $\overline{x}\in [-\frac{r}{4},0]$. Since 
$$  \psi (\overline{x}) \ge \psi(0), 
$$
we have  
 \begin{equation}\label{demontrer l inega;ite de delta v delta en 0 a partir du point max} 
-\a \overline{x}\le v^{\d}(\overline{x},\oo)-v^{\d}(0,\oo)+m_{\mu}(-y_{0},\overline{x},\oo) -m_{\mu}(-y_{0},0,\oo).	 
\end{equation} 
Now using the homogenization result of $m_{\mu }(-y_{0},z,\oo)$ (Theorem \ref{homo.du.pmj}), we have 
 $$  
 \sup\limits_{y \in [-r,0] }  \frac{1}{r}\left|m_{\mu}(-2r,y,\oo)+\overline{p}^{L,-}_{\mu}.(y_{0}+y)\right|=\sup\limits_{z \in [-1,0]  }\left|\frac{m_{\mu}(-2.r,z.r,\oo)}{r}+\overline{p}^{L,-}_{\mu}.(2+z)\right|\to 0,   
 $$
  as $r \to +\infty$. Then there exists $C$ such that for all $r \ge C$, we have 
 $$
   \sup\limits_{y \in [-r,0]  }  \frac{1}{r}\left|m_{\mu}(-y_{0},y,\oo)+\overline{p}^{L,-}_{\mu}.(y+y_{0})\right|\le \a^{2}.  
 $$ 
 The constant $C >1$ can been choose large enough, so we suppose that $C\ge 16 E$, with $E=\sup\limits_{y\in\R}|H(0,y,\oo)|$
  and we consider $r=\frac{C}{\a \d}\ge C>1$. We then have 
 $$
 \sup\limits_{y \in [-r,0] } \left|m_{\mu}(-y_{0},y,\oo)+\overline{p}^{L,-}_{\mu}(y+y_{0})\right|\le \a^{2}r \le \frac{C \a}{ \d}.  
 $$ 
 Injecting this in \eqref{demontrer l inega;ite de delta v delta en 0 a partir du point max}, we deduce that
\begin{align}\label{eq:v-delta1}
-\a \overline{x}\le& v^{\d}(\overline{x},\oo)-v^{\d}(0,\oo)-\overline{p}^{L,-}_{\mu} (\overline{x}+y_{0})+\overline{p}^{L,+}_{\mu}y_{0}+2 \a ^{2}r\nonumber \\
\le&v^{\d}(\overline{x},\oo)-v^{\d}(0,\oo)+2\alpha^{2}r\nonumber\\
\le& v^{\d}(\overline{x},\oo)-v^{\d}(0,\oo)+\frac{2C \a}{\d}.
\end{align}
Using that
$$ v^{\d}(\overline{x},\oo)-v^{\d}(0,\oo)\le \frac{2E}{\d} $$
we finally get 
$$-\a \overline{x}\le\frac{C}{4\d},
$$
and so
$$
0\ge \overline{x}\ge -\frac{r}{4}   
$$
 
 We now duplicate the variables, by defining the following function
$$
 \Phi_\theta(x,y)=v^{\d}(x,\oo)+m_{\mu}(-y_{0},y,\oo)+\a x-|x-\overline{x}|^{2} -\frac{|x-y|^{2}}{2\theta}. 
 $$
This function have a maximum denoted by $M_\theta$ and reached in $(x_{\theta},y_{\theta})\in [-r,0]\times [-r,0]$. 
We first claim that $x_{\theta} \to \overline{x}$ as $\theta \to 0$. Indeed, using the maximality of $(x_\theta, y_\theta)$, we have that
\begin{align*}
 \psi(\overline{x})+|x_\theta-\bar x|^2&=\Phi_\theta(\overline{x},\overline{x})+|x_\theta-\bar x|^2\\
 &\le \Phi_\theta(x_{\theta},y_{\theta})+|x_\theta-\bar x|^2\\
 & \le v^{\d}(x_{\theta},\oo)+ m_{\mu}(-y_{0},x_{\theta},\oo)+\a x_{\theta}+L_{\mu}|x_{\theta}-y_{\theta}|-\frac{|x_{\theta}-y_{\theta}|^{2}}{2 \theta}\\
 &\le\psi(\overline{x})+L_{\mu}^{2}\frac{\theta}{2}.
 \end{align*} 
This implies that $|x_\theta-\bar x|^2\le L_{\mu}^{2}\frac{\theta}{2}$ and proves the claim. 
We now claim that $\frac{|x_{\theta}-y_{\theta}|^{2}}{\theta}\to 0$ as $\theta \to 0$. First, by passing to the limit $\theta\to 0$ in the previous inequality, we get that $\Phi_\theta(x_{\theta},y_{\theta}) \to M=\psi(\overline{x}).$
Moreover, using that 
$$ 
M_{2\theta}\ge \Phi_{2\theta}(x_{\theta},y_{\theta})\ge \Phi_\theta(x_{\theta},y_{\theta})=M_{\theta}, 
$$
we get
$$ 0\le \Phi_{2\theta}(x_{\theta},y_{\theta})-\Phi_{\theta}(x_{\theta},y_{\theta}) =\frac{|x_{\theta}-y_{\theta}|^{2}}{4\theta}\le M_{2\theta}-M_{\theta}\to 0.
$$	
This implies that $\frac{|x_{\theta}-y_{\theta}|^{2}}{\theta}\to 0$ and  $(x_{\theta},y_{\theta}) \to (\overline x, \overline x)$. In particular, for 
$\theta$ small enough, $(x_{\theta},y_{\theta})\in ]-\frac{r}{2},0]\times ]-\frac{r}{2},0]$.

We now distinguish three cases\medskip

\noindent\textbf{Case 1:} $x_\theta, y_\theta<0$. 
We can then use the inequalities satisfied by $v^\d$ and $m_{\mu}(-y_0,\cdot,\oo)$.
We have that $v^{\d}(\cdot,\oo)$ is a solution of \eqref{correcteur app},  and $x_{\theta} $ is a maximum point of $v^{\d}(\cdot,\oo)-\phi_{1}(\cdot)$, where $\phi _{1}$ is defined by 
$$\phi_{1}( x)=- m_{\mu}(-y_{0},y_{\theta},\oo)-\a x+|x-\overline{x}|^{2}+\frac{|x-y_{\theta}|^{2}}{2\theta}.
$$
Hence 
$$ H \left(-\a+2(x_{\theta}-\overline{x})  +\frac{x_{\theta}-y_{\theta}}{\theta}, x_{\theta},\oo\right)  \le -\d v^{\d} (x_{\theta},\oo). 
 $$ 
 We  also have that $y \to n_{\mu}(y,-y_{0},\oo)=m_{\mu}(-y_{0},y,\oo)$ is a super-solution  in $]-\frac{r}{2},0[$ of 
 $$ 
 H\left(D(-m_{\mu}(-y_{0},y,\oo)),y,\oo\right)\ge \mu. 
 $$  
 Since $y_{\theta}$ is a minimum point of $ -m_{\mu}(-y_{0},\cdot,\oo)-\phi_{2} $, with $\phi_{2}$ defined by 
 $$
 \phi_{2}(y)=v^{\d}(x_{\theta},\oo)+\a x_{\theta}-|x_{\theta}-\overline{x}|^{2}
		-\frac{|x_{\theta}-y|^{2}}{2\theta},
$$ 
we obtain  that
$$ H \left(\frac{(x_{\theta}-y_{\theta})}{\theta},y_{\theta},\oo\right)  \ge \mu.
$$ 
Subtracting the two inequalities we get 
$$ 
H \left(-\alpha+2(x_{\theta}-\overline{x}) +\frac{(x_{\theta}-y_{\theta})}{\theta}, x_{\theta},\oo\right)-  H \left(\frac{x_{\theta}-y_{\theta}}{\theta},y_{\theta},\oo\right) \le -\mu-\d v^{\d}(x_{\theta},\oo).  
$$ 
Adding and subtracting the term $H \left(\frac{(x_{\theta}-y_{\theta})}{\theta},x_{\theta},\oo\right)$ and using the Lipschitz continuity of $H$ with respect to $p$, we obtain 
$$- c\a-2c|x_{\theta}-\overline{x}| \le  -\mu- \d v^{\d}(x_{\theta},\oo)+w\left(|x_{\theta}-y_{\theta}|\left(1+\left|\frac{x_{\theta}-y_{\theta}}{\theta}\right|\right)\right), 
$$
with $c$ the Lipschitz constant of $H$. Letting $\theta \to 0$, we get that 
$$
 \mu+ \d v^{\d}(\overline{x},\oo)=\overline{A}+\a+ \d v^{\d}(\overline{x},\oo)\le c\a
 $$ 
 and so
  $$\overline{A}+ \d v^{\d}(\overline{x},\oo)\le C' \a.$$
 To get this inequality in $0$, we use \eqref{eq:v-delta1} to get
 $$\overline{A}+ \d v^{\d}(0,\oo)\le\overline{A}+ \d v^{\d}(\overline{x},\oo)+\frac{2C\a}{\d}\le \left(C'+\frac{2C}\d\right)\a.
 $$
 \medskip

\noindent{\textbf{Case 2:}} $x_\theta=0, {y}_\theta<0 $. 
We have that $v^{\d}(\cdot,\oo)$ is a solution of \eqref{correcteur app}, then using Proposition \ref{restriction des solutions} and arguing in the same way as Case 1, we get
 $$ 
 H^{+} \left(-\a  -\frac{{y}_\theta}{\theta}, 0,\oo\right)  \le -\d v^{\d} (0,\oo).  
 $$ 
As in the previous case, we also have
$$ H \left(-\frac{{y}_\theta}{\theta},{y}_\theta,\oo\right)  \ge \mu.
$$
Using that (recall that $H^-$ is non-increasing)
$$ H^{-}\left(-\frac{{y}_\theta}{\theta},{y}_\theta,\oo\right)\le H^{-}(0,{y}_\theta,\oo)\le \overline A<\mu 
$$
and the fact that $H=\max(H^{-},H^{+})$, we get
$$ 
H^{+} \left(-\frac{{y}_\theta}{\theta},{y}_\theta,\oo\right)  \ge \mu. 
$$
Arguing as in Case 1 (with $H$ replaced by $H^+$), we get the result.
\medskip

\noindent\textbf{Case 3} $y_\theta=0$.  Using Proposition \ref{restriction des solutions} and arguing in the same way as Case 1, we get
$$ H^+ \left(\frac{x_\theta}{\theta},0,\oo\right)  \ge \mu.
$$
Using that $\frac {x_\theta}\theta\le 0$, we get
$$ \mu\le H^+ \left(\frac{x_\theta}{\theta},0,\oo\right)\le H^+(0,0,\oo)\le H(0,0,\omega)\le \overline A,$$
which contradicts the fact that $\mu>\overline A$ and so this case couldn't happen.
This ends the proof.

\end{proof}

\begin{corollary}\label{convergence du correct approch sur une boule infinie}
Let $\e_0$ small enough. Then, for all $r>0$, $0<\d \le \e\le \e_0$ and  $\oo\in \o_{0}^{\e_0}$, we have that 
$$
\d v^{\d}(y,\oo)\to- \overline{A} \mbox{ uniformly in }B_{r/\d}. 
$$
\end{corollary}
		
\begin{proof}Let $\oo \in \o_{0}^{\e_0}$. Since $\d v^{\d}$ is uniformly Lipschitz continuous, it's enough to show the convergence pointwise. 
Moreover, for $y\in \left[\frac {-2}{\sqrt \e},\frac2{\sqrt \e}\right]$, we have
$$ 
|\d v^{\d}(y,\oo )-\overline{A}| \le |\d v^{\d}(0,\oo)-\overline A|+|\d v^{\d}(y,\oo )-\d v^{\d}(0,\oo)|\le|\d v^{\d}(0,\oo)-\overline A| + C\d \frac2{\sqrt \e} \to 0.
$$
We now want to prove the convergence for $|y|>\frac2{\sqrt \e}$. We prove the convergence for $y\ge \frac2{\sqrt \e}$, the case $y \le -\frac2{\sqrt \e}$ being similar. Since $v^{\d}$ is  solution of \eqref{correcteur app} and using Proposition \ref{restriction des solutions}, we get, for all $y\ge \frac2{\sqrt \e}$ and $z \ge 0$ that
$v^{z}(y,\oo)=v^{\d}(y+z,\oo)$ is solution of 
$$
\left\{
\begin{aligned}
& \d w^{\d}+H_{R}(Dw^{\d},y+z,\oo )=0 &\mbox{ in }(\frac2{\sqrt \e},+\infty),\\
&\d w^{\d}+H_{R}^{-}(Dw^{\d},y+z,\oo )=0 &\mbox{ for }y=\frac2{\sqrt \e}.
\end{aligned}
\right.
$$
Moreover, using the stationarity of $H_{R}$, we have that $v^{\d}(y,\tau_{z}\oo)$ satisfies the same equation. Hence, the comparison principle implies that 
$$
v^{\d}(y+z,\oo)=v^{\d}(y,\tau_{z}\oo)\quad {\rm for}\;  y\ge \frac2{\sqrt \e} \; {\rm and}\; z\ge 0. 
$$
We now fix $\overline{\rho}_{0}= \frac2{\sqrt \e}$. For all $\oo\in \o_{0}^{\e_0}$, we have $\d v^{\d}(\overline{\rho}_{0},\oo )\to -\overline{A}.$
Let $y\in [\overline{\rho}_{0},r/\d] $. If $\tau _{y-\overline{\rho}_{0}}\oo \in \o_{0}^{\e_0}$, then by stationarity we have 
$$ 
v^{\d}(y,\oo)=v^{\d}(\overline{\rho}_{0}+y-\overline{\rho}_{0},\oo)=v^{\d}(\overline{\rho}_{0},\tau _{y-\overline{\rho}_{0}}\oo)\to -\overline{A}. 
$$
On the contrary, if $\tau _{y-\overline{\rho}_{0}}\oo \not\in \o_{0}^{\e_0}$, taking $\d$ small enough, there exists $z\ge \overline \rho_0$ such that $|y-z|\le \lambda \frac{r}{\d}$, with $\lambda>0$, such that $ \tau _{z-\overline{\rho}_{0}}\oo \in \o_{0}^{\e_0}$. We then have 
$$   
\d v^{\d}(y,\oo)=\d v^{\d}(y,\oo)-\d v^{\d}(z,\oo)+\d v^{\d}(\overline{\rho}_{0},\tau_{z-\overline{\rho}_{0}}\oo).
$$ 
Passing to the limit $\lambda,~\d \to 0$, we get 
$$|\d v^{\d}(y,\oo) -\overline{A}|\le |\d v^{\d}(y,\oo)-\d v^{\d}(z,\oo)|+|\d v^{\d}(\overline{\rho}_{0},\tau_{z-\overline{\rho}_{0}}\oo)-\overline A|\le C\lambda r.
$$
Taking $\lambda \to 0$, we get the result.

 \end{proof}
\subsection{Control of slopes}\label{resultat sur le domaine tronquee}
In this subsection, we prove that the rescaled approximated corrector verifies the good slopes at infinity \eqref{eq:W}. 
\begin{prop}[Control of slopes]\label{preuve des bonnes pentes du correc approch sur un domaine tronque}
Assume ${\bf (H}$-${\bf \e)}$ and $\overline{A}> \mu_{R}^{\star}$ and let $\e_0>0$ and $r>0$. Then there exists  $\d_{0}>0$, $\gamma_{0}>0$ such that there exists $C>0$ such that for all $\d \le \d_{0}$, $\d\le \e\le \e_0$, $\gamma\le \gamma_{0}$, $y> \frac 2{\sqrt \e}$, $h\ge0$ with $y, y+h\in B_{1/\d}$ and $\oo\in \o_{0}^{\e_0}$, we have that 
\begin{equation}\label{eq:16}
 v^{\d}(y+h,\oo)-v^{\d}(y,\oo)\ge (\overline{p}_{\overline A}^{R,+}-\gamma)h-C. 
\end{equation}
On the other side, if we  assume that $\overline{A}>\mu_{L}^{\star}$, then there exists  $\d_{0}>0$, $\gamma_{0}>0$ such that there exists $C>0$ such that for all $\d \le \d_{0}$, $\d\le \e\le \e_0$, $\gamma\le \gamma_{0}$, $y\le \frac{-2}{\sqrt \e} $, $h\ge0$  with $y, y-h\in B_{1/\d}$ and  $\oo\in \o_{0}^{\e_0}$, we have that 
\begin{equation}\label{eq:17}
v^{\d}(y-h,\oo)-v^{\d}(y,\oo)\ge (-\overline{p}_{\overline A}^{L,-}-\gamma)h-C. 
\end{equation}
\end{prop}
Before giving the proof of this proposition we need a comparison principle on bounded intervall. The proof of this comparison principle is given in \cite[Proposition 4.1]{Galise}

\begin{prop}[Comparison principle on bounded domain]\label{comparison.principale.tron.domaine}
Assume ${\bf (H}$-${\bf \e)}$ and let $\rho_2>\rho_1> \frac 2 {\sqrt \e}$ (resp. $\rho_2<\rho_1< \frac{-2}{\sqrt \e}$) and $ \lambda \in \R$. Let $v$ be a super-solution of the following problem 
$$\left\{
\begin{aligned}
&H(Dv,y,\oo)  \ge \lambda  &y\in (\rho_1,\rho_2),\\
&H^+(Dv,\rho_2,\oo)  \ge \lambda \\
&v(\rho_1,\oo)\ge 0,
\end{aligned}
\right.  
$$
 and let $u$ be a sub-solution, for some $\tilde \eps_0>0$ of  
$$ \left\{
\begin{aligned}
&H(Du,y,\oo)  \le \lambda-\tilde \eps_0  &y\in (\rho_1,\rho_2),\\
&H^+(Du,\rho_2,\oo)  \le \lambda-\tilde \eps_0\\
&u(\rho_1,\oo)\le 0.\\
\end{aligned}
\right. 
 $$ 
 Then 
 $$ u \le v~\mbox{in}~[\rho_1,\rho_2]. $$
		\end{prop}

\begin{proof}[Proof of Proposition \ref{preuve des bonnes pentes du correc approch sur un domaine tronque}.] We prove only the case $A>\mu^\star_R$, the other one being similar.
	Let $y_0\in [\frac 2 {\sqrt \e},+\infty)$, $r>0$ and $\tilde \eps_{0}>0$ small enough be such that $\overline{A}-2\tilde \eps_{0}>\mu_{R}^{\star}$. For $\oo \in \o_{0}^{\e_0}$, by Corollary \ref{convergence du correct approch sur une boule infinie}, there exists $\d_{0}$ such that for all $\d \le \d_{0}$, $$|\d v^{\d}(y,\oo)+\overline{A}|\le\tilde  \eps_{0}\mbox{ for all }y \in B_{r/\d}.$$
Since $v^{\d}(\cdot,\oo)$ is solution of \eqref{correcteur app} in $\R$, we then have 
$$   
 H(Dv^{\d},y,\oo)=-\d v^{\d}(y,\oo)\ge -\tilde \eps_{0}+\overline{A} \mbox{ in } B_{r/\d}.
 $$
 We set $v(y,\oo)=v^{\d}(y,\oo)-v^{\d}(y_{0},\oo)$, then by Proposition \ref{restriction des solutions} and using the fact that $H=H_{R}$ in $(\frac 2 {\sqrt \e},+\infty)$, we have that $v$ is solution of
 $$ \left\{
\begin{aligned}
&H_{R}(Dv,y,\oo)\ge\overline{A}-\tilde \eps_{0} & \mbox{ in }(y_{0},r/\d),\\
&H_{R}^+(Dv,r/\d,\oo)\ge \overline{A}-\tilde \eps_{0} \\
&v(y_{0})\ge0.
\end{aligned}\right.$$
 
 We now construct a strict super-solution of the same equation. 
 We set  $\mu=\overline{A}-2\tilde \eps_0>\mu_{R}^{\star}$ and $m(y,\oo)=m^R_\mu(y,0,\oo)-m^R_\mu(y_0,0,\oo)$, where $m^R_\mu$ is given in Proposition \ref{pro:3.1}. Then by Proposition \ref{restriction des solutions}, $m$ is solution of 

$$ \left\{
\begin{aligned}
&H_{R}(Dm,y,\oo)=\overline{A}-2\tilde \eps_{0} & \mbox{ in }(y_{0},r/\d),\\
&H_{R}^+(Dm,r/\d,\oo)= \overline{A}-2\tilde \eps_{0} \\
&m(y_0,\oo)=0.
\end{aligned}\right.$$
Then, the comparison principle, Proposition \ref{comparison.principale.tron.domaine}, yields that for all $y \in [y_{0},r/\d]$, 
$$ 
v(y,\oo) \ge m(y,\oo). 
$$ 
That is 
\begin{equation}\label{eq:15}
 v^{\d}(y,\oo)-v^{\d}(y_{0},\oo)\ge m^R_\mu(y,0,\oo)-m^R_\mu(y_0,0,\oo).
 \end{equation}
Using that, by Proposition \ref{pro:3.3} (recall that Lemma \ref{lem:3.5} gives that $\overline m_\mu^R(y)=\overline p^{R,+}_\mu.y$ for $y\ge 0$),
$$\sup_{B_{r/\d}}\frac 1\d\left | m^R_\mu(y,0,\oo)-\overline p^{R,+}_\mu\cdot y\right |=\sup_{B_{r}} \left | \frac {m^R_\mu\left(\frac z\d ,0,\oo\right)}\d-\overline p^{R,+}_\mu\cdot z\right |\to 0,
$$
we get that there exists a constant $C$ such that for all $\delta\le \delta_0$, we have
$$
\sup_{B_{r/\d}}\left | m^R_\mu(y,0,\oo)-\overline p^{R,+}_\mu\cdot y\right |\le C\d.
$$
Injecting this in \eqref{eq:15}, we get
$$
 v^{\d}(y,\oo)-v^{\d}(y_{0},\oo)\ge \overline p^{R,+}_\mu(y-y_0)-2 C_\d\quad \forall y,y_0\in \left(\frac{2}{\sqrt{\eps}},r/\d\right).
 $$

Finally if we choose $\gamma_{0}$, such that $\overline{p}_{\overline{A}}^{R,+}>\gamma_{0}\ge 0$ ($\mu_{R}^{\star}=\overline{H}_{R}(0)$), then for all $\gamma\le \gamma_{0}$, we have that 
$$  
 \overline{H}_{R}\left(\overline{p}_{\overline A}^{R,+}-\gamma\right)=\overline{H}_{R}^{+}\left(\overline{p}_{\overline A}^{R,+}-\gamma\right),
 $$ 
 and we can choose $\tilde \eps_{0}>0$ such that 
 $$ 
 \overline{p}_{\mu}^{R,+}=\overline{p}_{\overline A}^{R,+}-\gamma.  
 $$ 
This ends the proof. 
		
\end{proof}
		
\subsection{Proof of Theorem \ref{convergence de correcteur approche dans une boule et quil verifie les bonnes pentes}}
\begin{proof}[Proof of Theorem \ref{convergence de correcteur approche dans une boule et quil verifie les bonnes pentes}]
The existence of the approximated correctors is given in Proposition \ref{pro:existence-vd} while the convergence of $\d v^\d$ to $-\overline A$ in balls of radius $r/\d$ is given in Corollary \ref{convergence du correct approch sur une boule infinie}.  It just remains to show that $\overline v^\d$ converges to $W$ satisfying \eqref{eq:W}. 

First, we have that $\overline v^\d$ is Lipschitz continuous, uniformly in $\d$. Then, up to a subsequence, there exists $W$, with $W(0)=0$, such that
$$\overline v^\d\to W\quad\textrm{locally uniformly as }\d\to 0.$$
Like in \cite{Galise} and arguing as in the proof of convergence away from the junction point, we have that $W$ satisfies
\begin{align*}
&\overline H_R (W_y)=\overline A \textrm{ for }y>0\\
&\overline H_L (W_y)=\overline A \textrm{ for }y<0
\end{align*}
In the case $\overline A=\mu^\star_R$, this implies that 
$$\overline p_R^+\le W_y\le \hat p_R^+.$$
In the case $\overline A>\mu^\star_R$, we also get, from \eqref{eq:16}, that
$$W_y\ge  \overline p_R^+-\g$$
and so the equation satisfied by $W$ also implies that
$$\overline p_R^+\le W_y\le \hat p_R^+$$
and gives the particular form \eqref{eq:W} of $W$ for $y>0$.
Similarly, we can prove for $y<0$ that
$$ \hat p_L^-\le W_y\le\overline  p_L^-$$
which gives \eqref{eq:W} for $y<0$. This ends the proof.
\end{proof}

\section{Convergence result of the rescaled problem \eqref{P.en.epsilon}}\label{converg resultat}
We have defined the right and left deterministic hamiltonians and the effective flux limiter $\overline{A}$, so to complete the proof of Theorem \ref{Theoreme d homog du prob en eps}, we should prove the convergence result. This is the goal of this section.
Under Assumption ${\bf (H}$-${\bf \e)}$, we define $\o_0$ by 
$$\o_0=\bigcup_\e\o_0^\e.$$

We start by the definition of the relaxed half-limits, defined for all locally bounded family $\{u^{\eps}\}_{\eps}$, and for $\oo\in \o_{0}, \; x\in \R$, by
$$\left\{
\begin{aligned}
&\overline{u}(t,x,\oo)=\limsup\limits_{s\to t, y \to x,\eps \to 0}u^{\eps}(s,y,\oo),\\
&\underline{u}(t,x,\oo)=\liminf\limits_{s\to t, y \to x,\eps \to 0}u^{\eps}(s,y,\oo).
\end{aligned}
\right.  
 $$
To define these half-relaxed limits for the solution $u^\eps$ of \eqref{P.en.epsilon}, we have to get some bound.

\begin{lem}[Barriers]\label{barriere}Under Assumptions ${\bf (H)}$, and for all $\oo \in \o$, there exists a constant $C > 0$, independant of $\oo$, such that for all $\eps>0$, we have
$$|u^{\eps}(t,x,\oo)-u_{0}(x)|\le Ct\,\,\,\, \mbox{for all}\,\, (t,x) \in (0,T)\times\R.$$ 
\end{lem}

\begin{proof} Let $C_{0}$ be the Lipschitz constant of $u_{0}$. We set $C=\sup\limits_{\substack{y \in \R ,\\|p|\le C_{0}}}|H(p,y,\oo)|$. $C$ is well defined and independant of $\oo$ by hypothesis ${\bf (H2)}$. Then $ u_{0}(x)\pm Ct  $ are super- and sub-solutions of \eqref{P.en.epsilon}. So the result follows by the comparison principle. This ends the proof of the lemma.
		
 \end{proof}

We are now able to give the proof of Theorem \ref{Theoreme d homog du prob en eps}

\begin{proof}[Proof of Theorem \ref{Theoreme d homog du prob en eps}.]  Let $\oo \in\o_{0}$. We give the proof in the case of assumption ${\bf (H}$-${\bf \eps)}$. There exists $\e_0>0$ such that $\oo\in \o_0^{\e_0}$. In the sequel, we work with $\e\le \e_0$ (so that $\tilde A(\oo)=\overline A$ for all $\e$). To prove that $u^{\eps}(\cdot,\cdot,\oo)$ converges locally uniformly to $u(\cdot,\cdot)$ on $(0,T)\times \R$, it's enough to show that $\overline u$ and $\underline u$ are respectively sub and super-solution of \eqref{P.jonction.limite}. Then the comparison principle yields that, $\overline{u} \le \underline{u} $ and since, by definition, we have $\underline{u} \le \overline{u} $, we will get 
$$\lim\limits_{\eps \rightarrow 0}u_{\eps}(t,x,\oo)=u(t,x).
$$\medskip

We show that $\overline{u}$ is a sub-solution of \eqref{P.jonction.limite}, the proof for the super-solution being similar. First, note that the initial condition is satisfied by Lemma \ref{barriere}. 
We argue by contradiction by assuming that there exists a test function $\phi \in C^{1}((0,T) \times \R)$ with $\overline{u}(\overline{t},\overline{x})=\phi(\overline{t},\overline{x})$ and
\begin{equation} \label{ineg}
(\overline{u}-\phi)(t,x)<(\overline{u}-\phi)(\overline{t},\overline{x})=0 \,\,\, \mbox{ for all} \,\,(t,x) \in B_{\overline{r}}(\overline{t},\overline{x})\setminus\{(\overline{t},\overline{x})\},
\end{equation}	
such that
$$\left\{\begin{aligned}
&\phi_{t}(\overline{t},\overline{x})+\overline{H}_{L}(D\phi(\overline{t},\overline{x})) = \theta >0, ~&\mbox{if}~\overline{x} &\in ]-\infty,0[,\\
&\phi_{t}(\overline{t},\overline{x})+\overline{H}_{R}(D\phi(\overline{t},\overline{x})) = \theta >0, ~&\mbox{if}~\overline{x} &\in ]0,+\infty[,\\
&\phi_{t}(\overline{t},0)+\max\left\{\overline{A}, \overline{H}_{L}^{+}(D\phi(\overline{t},0^{-}),\overline{H}_{R}^{-}(D\phi(\overline{t},0^{+})\right\} =\theta > 0~&\mbox{if} ~\overline{x}&=0.
\end{aligned}\right.
$$
We distinguish two cases: $\overline{x}=0$ or $\overline{x}\neq 0$.
\medskip

\noindent{\bf Case 1:} $\overline{x}=0$. According to \cite{Imbert}, the function $ \phi $ should  has the following form
\begin{equation}\label{eq:20}
\phi(t,x)=\psi(t)+\overline{p}_{L}^- x \mathbf{1} _{\{x<0\}}+\overline{p}_{R}^+ x \mathbf{1} _{\{x>0\}} ,
\end{equation}
 with $\psi \in C^{1}(0,+\infty)$. 
\begin{rem}
For the super-solution, the test function has the following form:
$$\phi(t,x)=\psi(t)+\hat{p}_{L}^- x \mathbf{1} _{\{x<0\}}+\hat{p}_{R}^+ x \mathbf{1} _{\{x>0\}} .$$ 
\end{rem}
We then have
$$
\psi^{'}(\overline{t})  +\max\left\{\overline{A}, \overline{H}_{L}^{+}(D\phi(\overline{t},0^{-}),\overline{H}_{R}^{-}(D\phi(\overline{t},0^{+})\right\} =\psi^{'}(\overline{t})+\overline{A}=\theta > 0.
$$ 
We define the perturbed test function (see \cite{Evans}) by
$$\phi^{\eps}(t,x)=\psi(t)+\eps v^{\eps}\left(\frac{x}{\eps},\oo\right).
$$  
We claim that $\phi^{\eps}$ is a super-solution of
\begin{equation}\label{pert}
\phi^{\eps}_{t}+H\left(D\phi^{\eps},\frac{x}{\eps},\oo\right)=\frac{\theta}{2}\,\,\,\, \mbox{in} \,\,B_{r}(\overline{t},0),
\end{equation} 
for $r$ small enough. To prove this, let $\eta$ be a test function, such that $(t_{1},x_{1})$ is a minimum point of $\phi^{\eps}-\eta$ with $\phi^{\eps}(t_{1},x_{1})=\eta(t_{1},x_{1})$. 
We have
$$\phi^{\eps}(t,x) \ge \eta(t,x),
$$
i.e
$$\psi(t)+\eps v^{\eps}\left(\frac{x}{\eps},\oo\right) \ge \eta(t,x)
$$
which implies
$$ 
v^{\eps}\left(\frac{x}{\eps},\oo\right) \ge \frac{1}{\eps}\left(\eta(t,x)-\psi(t)\right).
$$
Using that 
$$
v^{\eps}(y_{1},\oo)= \frac{1}{\eps} \left(\eta( t_{1},\e y_{1})-\psi( t_{1})\right), 
$$ 
with $ y_{1}=\frac{x_{1}}{\eps}$, we deduce that 
$$
(t,y) \to v^{\eps}(y,\oo)-\left(\frac{1}{\eps} \left(\eta(t,\eps y)-\psi(t)\right)\right) 
$$ 
reaches a minimum at $(t_{1},y_{1})$. We then have $\eta_{t}(t_{1},x_{1})=\psi'(t_{1})$. Moreover, since  $v^{\e} $ is as super-solution of \eqref{correcteur app}, we have, for $\eps$ small enough,
$$
H\left(D\eta(t_1,y_1),\frac {x_1}\eps,\oo\right)\ge - \eps v^\eps\left(\frac {x_1}\eps,\oo\right)\ge \overline A -\frac \theta 4,
$$
where we have used that
$$ \eps v^{\eps}(x,\oo)\to -\overline{A}$$ 
in balls of radius $\frac{r}{\eps}$.
We then get
$$\eta_{t}(t_{1},x_{1})+H\left(D\eta(t_{1},x_{1}),\frac{x_{1}}{\eps},\oo\right) \ge \psi'(t_{1})+\overline{A}-\frac{\theta}{4}\ge\frac{\theta}{2}$$
for $r$ small enough.
Combining \eqref{eq:W} with \eqref{ineg} and \eqref{eq:20}, we can fix $\kappa_r>0$ and $\eps$ small enough such that
$$ 
u^{\eps}+\kappa_{r} \le \phi^{\eps}\,\,\,\,\mbox{on}\,\, \partial B_{r}(\overline{t},0).
$$
By the comparison principle, the previous inequality holds in  $B_{r}(\overline{t},0)$. Passing to the limit $\eps\to 0$ and $(t,x)\to (\overline t, \overline x)$, we get the following contradiction
$$ 
\overline{u}(\overline{t},0)+\kappa_{r} \le \phi(\overline{t},0)=\overline{u}(\overline{t},0).
$$

\medskip

\noindent{\bf Case 2 :}  $\overline{x} \ne 0$. We use the approximate corrector problem far from the junction i.e given by the right and left hamiltonian. The proof is classical and we skip it.
			
This ends the proof of the theorem.

\end{proof}

\vspace{10mm}
\noindent{\bf ACKNOWLEDGMENTS}

 The authors would like to thanks P. Cardaliaguet and P. Calka for fruitful discussions in the preparation of this work.
This project was co-financed by the European Union with the European regional development fund (ERDF,18P03390/18E01750/18P02733) and by the Normandie Regional Council via the M2SiNUM project and by ANR MFG (ANR-16-CE40-0015-01).

	\end{document}